\documentclass[12pt,a4paper]{article}
\newcommand{\CC}{\mathbb{C}}
\newcommand{\RR}{\mathbb{R}}
\newcommand{\ZZ}{\mathbb{Z}}

\def\Log{{\rm Log\,}} 
\usepackage{amssymb}
\usepackage[all,cmtip]{xy}
\usepackage{amsmath}
\usepackage{eufrak}
\usepackage{stmaryrd}
\usepackage{mathrsfs}
\usepackage{latexsym}
\usepackage{amsfonts}
\usepackage{amsthm}
\usepackage{bbm}
\usepackage{dsfont}

\def\R{\mathbb{R}}
\def\C{\mathbb{C}}
\def\Z{\mathbb{Z}}

\def\N{\mathbb{N}}
\def\T{\mathbb{T}}

\newtheorem{lemma}{Lemma}[section]
\newtheorem{theorem}[lemma]{Theorem}

\newtheorem{corollary}[lemma]{Corollary}
\newtheorem{notation}[lemma]{Notation}
\newtheorem{definition}[lemma]{Definition}
\newtheorem{remark}[lemma]{Remark}
\newtheorem{example} [lemma]{Example}
\newtheorem{problem} [lemma]{Problem}
\newcommand{\Tnp}{\mathscr{T}_n^p}

\newcommand{\cD}{\mathcal{D}}

\begin{document}

\title{Complex Tropical Currents, Extremality, and Approximations
}
\author {Farhad Babaee}
\date{}
\maketitle

\begin {abstract}
To a tropical $p$-cycle $V_{\mathbb{T}}$ in $\mathbb{R}^n$, we naturally associate a normal closed and $(p,p)$-dimensional current on $(\mathbb{C}^*)^n$ denoted by $\mathscr{T}_n^p(V_{\mathbb{T}})$. Such a ``tropical current'' $\mathscr{T}_n^p(V_{\mathbb{T}})$ will not be an integration current along any analytic set, since its support has the form ${\rm Log\,}^{-1}(V_{\mathbb{T}})\subset (\mathbb{C}^*)^n$, where ${\rm Log\,}$ is the coordinate-wise valuation with $\log(|.|)$. We remark that tropical currents can be used to deduce an intersection theory for effective tropical cycles. Furthermore, we provide sufficient (local) conditions on tropical $p$-cycles  such that their associated tropical currents are ``strongly extremal'' in $\mathcal{D}'_{p,p}((\mathbb{C}^*)^n)$. In particular, if these conditions hold for the effective cycles, then the associated currents are extremal in the cone of strongly positive closed currents of bidimension $(p,p)$ on $(\mathbb{C}^*)^n$. Finally, we explain certain relations between approximation problems of tropical cycles by amoebas of algebraic cycles and  approximations of the associated currents by positive multiples of integration currents along analytic cycles. 
\end{abstract}

\section{Introduction}
A positive closed current $T$ on a complex manifold $X$ is called extremal in the cone of closed positive currents, if any decomposition of $T$ into a sum of two non-zero positive closed currents, $T=T_1 + T_2$~, implies that $T_1$ and $T_2$ are positive multiples of $T$ (see Section 2 for an introduction to currents). It was noted in  \cite{DemEx} that the cone of positive closed $(p,p)$-currents is the closed convex envelope of the extremal elements of this cone (endowed the weak topology of currents). Pierre Lelong in \cite{Lelong} showed that integration currents along irreducible analytic cycles are extremal, and  also asked whether positive multiples of integration currents along irreducible analytic cycles are the only extremal currents. Subsequently in \cite{DemEx}, Jean-Pierre Demailly found an example of an extremal current in $\mathbb{CP}^2$, namely $T_D:=dd^c\, \log \max \{|z_0|,|z_1|,|z_2| \}$, which has a support of real dimension $3$, and therefore cannot be an integration current along any analytic set. Later on, Eric Bedford noticed that many extremal currents naturally occur in dynamical systems on several complex variables whose supports are in general fractal sets, and therefore not analytic (see \cite{sibony}, \cite{Dinh-Sibony}, \cite{Guedj}, \cite{sib-dinh-rigidity} and references therein). 
\vskip 1mm 
Here we attempt to generalize Demailly's example, using the fact that $T_D$ has the extremality property, but in a stronger sense~: that for every other normal closed current $\tilde{T}$ of bidimension $(1,1)$ on $\mathbb{CP}^2$, which has the same support as $T_D$ there exists a $\rho \in \CC$ such that $\tilde{T} = \rho \, T_D\,$. Evidently, this \textbf{strong extremality} is a property of the supports. Defining $\Log: (\CC^*)^n \rightarrow \RR^n$, $(z_1,\dots,z_n) \mapsto (\log|z_1|,\dots,\log|z_n|)$, we note that for $n=2$, the restriction of $T_D$  to $(\CC^*)^2$ is just the closed positive current of bidimension $(1,1)$ given by $dd^c\, \log \max \{1,|z_1|,|z_2| \}$. Thus, the support of this current is $\Log^{-1}(L_{\mathbb{T}})$, where $L_{\mathbb{T}}$ is a tropical line in $\RR^2$. Accordingly, we attach to any tropical $p$-cycle ${V}_{\mathbb{T}}\subset \RR^n$ (Definition \ref{trop-cyc-def}) a normal closed current of bidimension $(p,p)$ with support equal to $\Log^{-1}(V_{\mathbb{T}})\subset (\CC^*)^n$. We refer to such a current as the \textbf{tropical current} attached to the tropical $p$-cycle $V_{\mathbb{T}}\subset \mathbb R^n$, denoted by $\mathscr T_n^p (V_{\mathbb T})$. Such a construction could actually be carried out for any weighted polyhedral $p$-dimensional complex $\mathcal P$ and one will prove that if $\mathcal P$ fulfills the balancing condition at each {facet} (\textit{i.e.} face of codimension one), then closedness of $\mathscr T_n^p (\mathcal P)$ is implied (Theorem \ref{maintheohigherdim}). Moreover, extremality of the tropical currents in the stronger sense is detectable from the combinatorial data of the corresponding tropical cycles~; suppose that a tropical $p$-cycle $V_{\T}\subset \RR^n$ is connected in codimension $1$. Also that at each facet $W,$ a set of primitive vectors $\{v_1,\dots,v_s \}$ which makes the balancing condition hold at $W$ satisfies the following two conditions~: first $\{h_W(v_1),\dots,h_W(v_s)\}$, spans the dual space $W^{\perp}$ as an $\RR$-basis, where $h_W$ is the projection along $W$; second, every proper subset of $\{h_W(v_1),\dots,h_W(v_s)\}$ is a set of independent vectors~; then the current $\mathscr T_n^p(V_{\mathbb T})$ stands as a strongly extremal element of $(p,p)$-dimensional normal closed currents on $(\CC^*)^n$, (Theorem \ref{maintheohigherdim}).\\

Employing tropical geometry, we will also show that tropical currents associated to effective tropical hypersurfaces are of the form $dd^c \,[q \circ \Log]\,,$ 
where  $q: \RR^n\rightarrow \RR$ is a tropical polynomial and,
$$
dd^c \,[q \circ \Log] = \mathscr{T}^p_n(V_{\T}(q))\,, 
$$
where $V_{\T}(q)$ is the associated (effective) tropical hypersurface (Theorem \ref{thm-hyper-ddc}). We remark that using a formula of Alexander Rashkovskii in \cite{Rash} and expansion of Monge-Amp\`ere measure of a tropical polynomial (as in Example 3.20 in \cite{yger}), the above relation provides an intersection theory for the effective tropical cycles (Remark \ref{rem-intersection}). 

\vskip 2mm
In addition we prove (Theorem \ref{current-amoeba-approx}) that if for a $V_{\mathbb{T}}$, $\mathscr{T}_n^p(V_\mathbb{T})$ is positive and extremal, and $V_{\mathbb{T}}$ is \textbf{set-wise} approximable (Definition \ref{set-wise-def}) by amoebas of algebraic cycles in $(\CC^*)^n$, then   $\mathscr{T}_n^p(V_\mathbb{T})$ is in the closure (in the sense of currents) of 
$$
\mathcal{I}^p((\CC^*)^n)= \big\{\lambda[Z]: \, Z\ \text{ $p$-dim. irreducible analytic subset in } (\CC^*)^n, \, \lambda \geq 0 \big\}.
$$ 

\vskip 2mm
This article is structured as follows. In Section 2, we state the preliminaries of the theory of currents. Section 3 is a brief discussion of tropical geometry. In Section 
$4$, we define the tropical current 
$\mathscr T_n^p(\mathcal P)$ attached to a weighted $p$-dimensional polyhedral complex $\mathcal P\,.$ In this section, we state and prove the main result of this paper about extremality (mentioned above)~; this is done step by step, first in the case $p=1$, then in the case $p>1$ assuming $\mathcal P$ has a single facet and then finally in the general case of $p>1$ tropical cycles. In Section 5, we illustrate how these constructions may be used in order to produce extremal currents on complex projective planes, and we remark how an intersection theory can be deduced. In Section 6, we discuss the problems of the approximability of tropical cycles by amoebas and will explain how it could be related to the problem of approximating the tropical currents by analytic cycles. We end this paper with some open problems. 

\section{Currents}\label{currents-section}
Throughout this paper $X$ is either $(\CC^*)^n$, $\CC^n$ or $\mathbb{CP}^n,$ which are analytic complex manifolds of dimension $n$. If $k,p,q$ are non-negative integers, possibly $k=\infty$, we denote by $\mathcal{C}^k_{p,q}(X)$ (resp. $\mathcal{D}^k_{p,q}(X)$) the space of differential forms of bidegree $(p,q)$ and of class $C^k$ (resp. with compact support) on $X$. The elements of $\mathcal{D}^k_{p,q}(X)$ are  called test forms. 
\vskip 1mm
The space of currents of \textbf{order} $k$ and of bidimension $(p,q)$, or equivalently of bidegree $(n-p,n-q),$ is by definition the topological dual space $[\mathcal{D}^k_{p,q}(X)]'$, where $\mathcal{D}^k_{p,q}(X)$ is endowed with the inductive limit topology.  $\mathcal{D}_{p,q}^{\infty}(X)$ (resp. $[\mathcal{D}_{p,q}^{\infty}(X)]'$) is usually denoted instead by $\mathcal{D}_{p,q}(X)$ (resp. by $\mathcal{D}'_{p,q}(X)$). A current $T \in \mathcal{D}'_{p,q}(X)$ is called \textbf{closed} if for every $\alpha \in \mathcal{D}_{p-1,q}(X)$,
\begin{equation}\label{def-closed}
\langle dT , \alpha \rangle := (-1)^{p+q+1} \langle T, d\alpha \rangle \,,
\end{equation}
vanishes. 
\vskip 2mm
 An important concept in this theory is {positivity}.
\begin{definition} \rm
 A form $\psi \in \mathcal{C}_{p,p}^{0}(X)$ is called 
\begin{itemize}
\item {strongly positive,} if for all $z \in X$, $\psi(z)$ is in the convex cone generated by $(p,p)$ forms of the type
$$
(i \psi_1 \wedge \bar{\psi}_1)\wedge \dots \wedge (i \psi_p \wedge \bar{\psi}_p),
$$  
where $\psi_j \in \bigwedge^{1,0} T^*_z X\,;$ 
\item positive, if at every point $z \in X$ and all $p$-planes $F$ of the tangent space $T_zX,$ the restriction $\psi(z)_{|F}$ is a strongly positive $(p,p)$-form. 
\end{itemize}
A current $T \in \cD'_{p,p}(X)$ is called strongly positive (resp. positive), if $$\langle T, \psi \rangle \geq 0$$ for every positive (resp.  strongly positive) test form $\psi \in \cD_{p,p}(X)$.  We denote the set of positive (resp. strongly positive) closed currents of bidimension $(p,p)$ by 
$$
PC^p(X), \quad (\text{resp. } SPC^{p}(X)).
$$
\end{definition}



In this paper we are mainly concerned with {\bf extremal} currents. Recall that the {\bf support} of a current is the smallest closed set in the ambient space $X$ such that on its complement the current vanishes, and a current $T$ is called \textbf{normal} if $T$ and $d T$ are of order zero. One can see that every closed positive current is normal.

\begin{definition}\label{extremality-def}
\rm
A current $T \in PC^p(X)$ (resp. $\in SPC^p(X)$) is called extremal in $PC^p(X)$ (resp. in $SPC^p(X)$) if whenever we have a decomposition $T = T_1 + T_2$ with $T_1, T_2 \in PC^p(X)$ (resp. $\in SPC^p(X)$), then there exist $\lambda_1, \lambda_2 \geq 0$ such that $T= \lambda_1 T_1$ and $T= \lambda_2 T_2$~. A closed current $T \in \mathcal{D}'_{p,p}(X)$ of order zero is called strongly extremal, if for any closed current $\tilde{T}\in \mathcal{D}'_{p,p}(X)$ of order zero which has the same support as $T$, there exists a $\rho \in \CC$ such that $T =\rho \, \tilde{T}.$
\end{definition}

\begin{remark}{\rm Note that the extremality properties are invariant under invertible affine linear transformations. Furthermore, strong extremality of a positive (resp. strongly positive) closed current $T \in \mathcal{D}'_{p,p}(X)$ implies extremality $PC^p(X)$ (resp. in  $SPC^p(X)$)~. In addition, strong extremality can be considered as a ``rigidity'' property of supports (see also \cite{sib-dinh-rigidity}). Therefore, if a normal closed $(p,p)$-dimensional current $T$ is supported by a set of the form $\Log^{-1}(V_{\mathbb{T}})\subset (\CC^*)^n$ for a tropical $p$-cycle $V_{\mathbb{T}}$, then the question of the strong extremality of $T$ relies merely on the combinatorial structure of $V_{\mathbb{T}}$. 
}
\end{remark}

Let us denote 
$$
\mathcal{I}^p(X)= \big\{\lambda [Z]: \  \lambda \geq 0, \, Z \subset X,\ \text{$p\,$-\,dimensional irreducible analytic subset }  \big\},
$$
and by $\mathcal{E}^p(X)$ the set of extremal elements of $SPC^p(X)$. Using the support theorems below, it is not hard to see that (\cite{Lelong}, \cite{Dembook}) 
$$
\mathcal{I}^p(X) \subset \mathcal{E}^p(X). 
$$

 
\subsection{Support theorems}

We need to quote two important structure theorems for supports of currents. For a through treatment see \cite{Dembook}.

\vskip 2mm
Let $S \subset X$ be a closed $C^1$ real submanifold of $X$  $(=\,(\CC^*)^n$, $\CC^n$ \textnormal{or} $\mathbb{CP}^n)$. The complex dimension of the holomorphic tangent in $S$, i.e. 
$$
{\rm dim}_{\CC} \, (T_xS \cap i T_xS) \, , 
$$
is called the \textbf{Cauchy-Riemann} dimension of $S$ at $x$. The maximal dimension 
$${\rm max}_{x \in S} \,{\rm dim}_{\CC} \, (T_xS \cap i T_xS)
$$
is called the Cauchy-Riemann dimension of $S$, denoted by ${\rm CRdim }\, S$. If this dimension is constant for all $x \in S$, then $S$ is called a Cauchy-Riemann submanifold of $X$. 
\vskip 2mm
The following theorem implies that a complex structure of dimension at least $p$ is needed on the support of a normal current in order to accommodate $(p,p)$ test forms.

\begin{theorem}[Theorem 2.10 in \cite{Dembook}]\label{supportsmall}
Suppose $T \in \cD'_{p,p}(X)$ is a normal current. If the support of $T$ is contained in a real submanifold $S$ of Cauchy-Riemann dimension less than $p$, then $T=0$.
\end{theorem}

\vskip 1mm

The next theorem about supports permits us to streamline a current if its support is a fiber space. 

\begin{theorem}[Theorem 2.13 in \cite{Dembook}]\label{supportdecom}
Let $S \subset X$ be a Cauchy-Riemann submanifold with Cauchy-Riemann dimension $p$ such that there is a submersion $\sigma: S \rightarrow Y$ of class $C^1$ whose fibers $\sigma^{-1}(y)$ are connected and that for all the points $z \in S$ we have 

$$
T_z S \cap i\, T_z S = T_z F_z,
$$
where $F_z = \sigma^{-1}(\sigma(z))$ is the fiber of the point $z$ and $T_zS$, $T_z F_z$ are the tangent spaces at $z$ corresponding to $S$ and $F_z$. Then, for every closed currents $T$ of bidimension $(p,p)$ and of order $0$ (resp. positive) with support in $S$, there exists a unique (resp. positive) Radon measure $\mu$ on $M$ such that 
$$
T = \int_{y \in Y} [\sigma^{-1}(y)]d\mu(y), 
$$
i.e. 
$$
\langle T, \psi \rangle = \int_{y \in Y}  \big(\int_{\sigma^{-1}(y)} \psi \big)d \mu(y) \, ,
$$
for $\psi \in \cD_{p,p}(X).$
\end{theorem}

\section{Tropical cycles}\label{tropicalsection}
We start off by recalling a definition of tropical curves. Throughout this article a \textbf{rational} graph is a finite union of rays and segments in $\RR^n$ whose directions have rational coordinates. We call these rays and segments ($1$-cells) as edges and the endpoints ($0$-cells) as vertices. Hence a graph $\Gamma$ is the data $(\mathcal{C}_0(\Gamma), \mathcal{C}_1(\Gamma))$ of the $0$-cells and $1$-cells. A \textbf{primitive} vector is an integral vector such that the greatest common divisor of its components is $1\,.$ For each edge $e$ incident to a vertex $a$ there exists a primitive vector $v_e$ which has a representative with support on $e$ pointing away from $a$. Assume that every edge $e$ of $\Gamma$ is weighted by a non-zero integer $m_e$\,. We say that $\Gamma$ satisfies the \textbf{balancing condition} at a vertex $a$ if 
\begin{equation}\label{balancingcondition}
\sum_{\{a\}\prec\, e\, \in\, \mathcal{C}_1(\Gamma)}m_e v_e =0,
\end{equation} 
where the sum is taken over all the edges incident to the vertex $a.$

\begin{definition}\rm
A tropical curve in $\RR^n$ is a weighted rational graph $\Gamma= (\mathcal{C}_0(\Gamma), \mathcal{C}_1(\Gamma))$ which satisfies the balancing condition (\ref{balancingcondition}) at every vertex $a \in \mathcal{C}_0(\Gamma)$. 
\end{definition}

In the same spirit, one can define the \textbf{tropical $p$-cycles} in $\RR^n$. First, a \textbf{polyhedral complex} is a finite set of polyhedra which are joined to each other along common faces. A polyhedral complex is called \textbf{rational} if each polyhedron is the intersection of rational half spaces, \textit{i.e.} the half spaces which are given by the inequalities of the form 
$$
\langle \nu , x \rangle \geq a\,, \quad \textnormal{with given constants } \nu \in \ZZ^n\,, a\in \RR^n \,,\quad \forall x\in \RR^n\,.$$
 Such a complex is said to be \textbf{weighted} if a non-zero integral weight is assigned to each of its $p$-dimensional cells. Now assume that a $(p-1)$-dimensional face $W$ is adjacent to $p$ dimensional faces $P_1, \dots, P_s~$, $s \geq 2$, which have respective weights $m_1,\dots, m_s\,.$ Choose a point $a$ in $W$ and respective primitive vectors $v_j$, $(j= 1, \dots, s),$ emanating from $a$ inward each $P_j\,$.  One defines the balancing condition in higher dimensions to be that the sum
\begin{equation}\label{balancing-higher-dim}
\sum_{j=1}^{s} m_j \, v_j \text{~~ is parallel to~~} W\,.
\end{equation}

\begin{remark}\label{remark-higher-balancing}
\rm{
Assume that $W$ lies in an affine $(p-1)$-plane $H_W$ and that each $V_j$ lies in an affine $p$-plane $H_{V_j}$. One can find a $\ZZ$-basis $\{w_1,\dots,w_{p-1} \}$ for $W\cap \ZZ^n$ (the initial point for these vectors is considered to be a point in $W$) and extend it to $\{w_1, \dots, w_{p-1}, v_j \}$ for each $j=1,\dots, s,$ such that $\{w_1, \dots, w_{p-1}, v_j \}$ is a $\ZZ$-basis for $H_{V_j} \cap \ZZ^n$ and the balancing condition (\ref{balancing-higher-dim}) is satisfied by the $v_j$. This simply implies that  $\sum_{j=1}^{s} m_j \, v_j$ lies in $H_W\,;$ in other words every $p\times p$ minor of the $n\times p$ matrix with columns vectors $\big(w_1, \dots, w_{p-1}, \sum_{j=1}^{s} m_j \, v_j\big)$ vanishes. 
}
\end{remark}

\begin{definition}[\cite{Mikh-trop-appli}, \cite{firststeps}, \cite{shaw}]\label{trop-cyc-def}\rm A weighted rational polyhedral complex of pure dimension $p$ is called a tropical $p$-cycle if the balancing condition (\ref{balancing-higher-dim}) is satisfied at every codimension $1$ face. Such a cycle is called effective if every weight is a positive integer.  
\end{definition}
Therefore, tropical $1$-cycles are the tropical graphs. Also, a tropical $(n-1)$-cycle in $\RR^n$ is called a tropical hypersurface. To define the effective tropical cycles of codimension $1$, one might use \textbf{tropical polynomials} which are defined as follows.
\begin{definition}\rm
A tropical Laurent polynomial $p:\RR^n\rightarrow \RR$ is a function of the form
\begin{equation}
(x_1,\dots,x_n)\mapsto \max\big\{c_{\alpha_1,\dots,\alpha_n}+ \alpha_1 x_1+\dots+\alpha_n x_n \big\},
\end{equation}
over a finite set of indices, in which $\alpha_i$, $i=1,\dots,n$ are integer numbers and $c_{\alpha_1,\dots,\alpha_n}$ are real numbers; we might abbreviate the notation to 
\begin{equation}
x \mapsto \max_{\alpha}\big\{c_{\alpha}+\langle\alpha, x \rangle \big\},
\end{equation}
where $\alpha=(\alpha_1,\dots,\alpha_n)$ and $\langle \ , \, \rangle$  is the usual inner product in $\RR^n$. 
\end{definition}
To justify the above definition one considers the \textbf{tropical semi-field} $(\mathbb{T},\oplus,\odot ).$ Where $\mathbb{T}= \RR \cup \{-\infty \}$, with the operations $a\oplus b = \max\{a,b \}$ and $a \odot b = a+b$ for $a,b \in \mathbb{T}.$ Then the usual definition of a  Laurent polynomial carried with tropical operations 
instead of the usual ones leads to that of a tropical Laurent polynomial in which the $c_{\alpha}$'s are coefficients and $\alpha_i$'s the respective degrees. If all $\alpha_i \in \mathbb{Z}^{\geq 0} ,$ the tropical Laurent polynomial $p$ is said to be a {tropical polynomial}. The \textbf{tropical hypersurface} corresponding to a given tropical Laurent polynomial $p$ is denoted by $V_{\mathbb{T}}(p)$ and defined as the set below 
$$
V_{\mathbb{T}}(p)\! =\! \big\{ x\in \RR^n : \textrm{values of at least two monomials in $p$ coincide and maximize at $x$} \big\},
$$
which is basically the corner locus of $x\mapsto p(x)$~: the set of points over which the graph of the piece-wise linear convex function $p$ is broken. This set has a rational polyhedral complex structure of pure dimension $n-1$. However, we still need to assign the weights to each of the polyhedra to make it an honest tropical cycle~: suppose $F$ is a $(n-1)$-dimensional cell where the monomials $c_{\alpha_j}+ \langle \alpha_j ,x \rangle$, $\alpha_j \in \mathbb{Z}^n$, $j=1,\dots,s$ are equal and maximized, then for dimensional reasons, the slopes $\alpha_j$ lie in a line in $\mathbb{Z}^n$~; the weight $w(F)$ assigned to $F$ is the maximal lattice length of this line segment connecting the points representing these slopes (the lattice length of a line segment being the number of lattice points on this line minus $1$)~; one can check that with such weights $V_{\mathbb{T}}(p)$ satisfies the balancing condition at each facet.  
\vskip 2mm

One also defines the \textbf{tropical projective space} in the following way. 
\begin{definition} \rm
The tropical projective space $\mathbb{TP}^n= \mathbb{T}^n \backslash \{(-\infty)^{n}\} \slash \sim$ where $\sim $ is the relation defined by 
$$
(x_0,\dots,x_n)\sim (\lambda \odot x_0,\dots,\lambda \odot x_n), \quad \lambda \in \mathbb{T}^* =\RR \, .
$$ 
\end{definition}
Now one considers the tropical cycles in $\mathbb{TP}^n$, which are locally the tropical cycles in open subsets of $\RR^n.$ Accordingly, the \textbf{homogeneous} tropical polynomials and the associated tropical hypersurfaces in $\mathbb{TP}^n$. The theorem below is now classic in tropical geometry.
\begin{theorem}[\cite{Mikh-pants}, \cite{firststeps} for $n=2$]\label{tropicalhyp}
Every effective tropical hypersurfaces in $\RR^n$ (resp. $\mathbb{TP}^n$) is of the form $V_{\mathbb{T}}(p)$, for a tropical (resp. homogeneous tropical) polynomial $p$.
\end{theorem}
\vskip 2mm
Another important notion to be recalled here is the notion of \textbf{amoeba}. Consider 
\begin{equation}
\Log_t: (\CC^*)^n \rightarrow \RR^n, \quad (z_1,\dots\, ,z_n) \mapsto (\frac{\log|z_1|}{\log t},\dots \, , \frac{\log|z_n|}{\log t})
\end{equation}
(when $t=\exp(1)$ we drop the subscript). 
 
\begin{definition}[\cite{Gelfand}]\rm
The amoeba of an algebraic subvariety $V \subset (\CC^*)^n$, denoted by $\mathcal{A}_V$, is the set $\Log(V)\subset \RR^n.$
\end{definition}

Given a family $(X_t)_{t \in \RR_{+}}$ of algebraic subvarieties of $(\CC^*)^n$, one considers the family of amoebas $\Log_t(X_t)\subset \RR^n$. Assume that $\Log_t(X_t)$, as $t$ goes to infinity, converges (with respect to the Hausdorff metrics on compact sets of $\RR^n$) to a limit set $X$~; then $X$ inherits a structure of a tropical cycle, {\it i.e.} as a set it is a rational polyhedral complex, which can moreover be equipped with positive integer weights to become balanced (see \cite{speyer thesis}). Therefore, a natural question arises~: which effective tropical cycles can be realized as Hausdorff limit of amoebas of a family of algebraic subvarieties of $(\CC^*)^n~?$ We are concerned with a modification of this problem in Section \ref{approximationsection} and we relate these approximations to approximations of tropical  extremal currents by analytic cycles. 


\section{Tropical currents}\label{tropicalcurrents}

Assume $V_{\mathbb{T}}$ is a tropical $p$-cycle. We define a current supported on $\Log^{-1}(V_{\mathbb{T}})$ which inherits the respective weights of $V_{\mathbb{T}}$ and then determine whether this current is strongly extremal. We introduce the following abridged notations. 
\begin{notation}
{For a complex number $\zeta $ and an integral vector $\nu =(\nu_1,\dots,\nu_m)$ ($m\in \mathbb N^*$) we set 
$$
\zeta^{\nu} = (\zeta^{\nu_1},\dots,\zeta^{\nu_m}).
$$
Moreover for two vectors $\nu = (\nu_1,\dots,\nu_m), \, \nu' = (\nu'_1,\dots,\nu'_m)$
$$
\nu \star \nu' := (\nu_1 \, \nu'_1, \dots, \nu_m \, \nu'_m).
$$}
\end{notation}

\vskip 2mm

Recall that a rational $p$-plane in $\RR^n$ means that is given by the equations 
$$
\langle \nu_i , x \rangle = 0, \quad \nu_i \in \ZZ^n\,, i\, =\, 1,\dots, n-p \,.  
$$
\begin{lemma}\label{section4-lemma1} 
Suppose $H$ is a rational $p$-plane in $\RR^n$ (which passes the origin), ($1\leq p\leq n$). Let $B= \big(w_1, \dots, w_p \big)$ and $B'= \big(w'_1, \dots, w'_p \big)$ be two $\ZZ$-basis for $H\, \cap \,\ZZ^n$. Define for any 
$\gamma\in (\mathbb S^1)^n$, the two subsets of $(\CC^*)^n$~: 
$$
Z_B^\gamma :=\{\tau_1^{w_1} \star \dots \star \tau_p^{w_p} \star \gamma =
 \iota^{w}_\gamma  (\tau)~;  \tau_1,\dots, \tau_p \in \CC^* \}
$$
and 
$$
Z_{B'}^\gamma = \{\tau_1^{w'_1} \star \dots \star \tau_p^{w'_p} \star \gamma =\iota^{w'}_\gamma(\tau)~;  \tau_1,\dots, \tau_p \in \CC^* \}.
$$
Then, the integration currents 
$$
T= [Z_B^\gamma]:=(\iota^w_\gamma)_* ([(\CC^*)^p]),\quad 
T'=[Z_{B'}^\gamma]:=(\iota^{w'}_\gamma)_* [(\CC^*)^p]
$$
coincide. 
\end{lemma}
\begin{proof}
The analytic sets $Z_B^\gamma$ and $Z_{B'}^\gamma$ are equal. 
We prove that they are analytically isomorphic. 
Consider $B, \, B'$ as matrices with the given vectors as columns. There exists $C \in GL(p, \ZZ),$ such that $B\,C=B'$. Set 
$$(\tau'_1,\dots , \tau'_p) = (\tau_1^{c_1} \star \dots \star \tau_p^{c_p} )$$
 where $c_1,\dots, c_p$ are the columns of $C$. This is an invertible monoidal change of coordinates, and it is easy to see that 
$$
(\tau'_1)^{w_1} \star \dots \star (\tau'_p)^{w_p}= \tau_1^{w'_1}\star \dots \star  \tau_p^{w'_p}\,,
$$
which concludes the proof. 
\end{proof}

\begin{remark}\label{remarktoricset}
\textnormal{
The sets of the form $Z_B^\gamma$, when $\gamma =1$, are referred to as \textbf{toric} sets \cite{sturm-toricset}. 
They can be understood as zero locus of binomial ideals 
in $\C^n$. In fact, if $\xi_1,..., \xi_M$ is a set of primitive generators for 
${\rm Ker}\, B^t \cap \mathbb Z^n$ , such that each $\xi_\ell$ 
is split as $\xi_\ell^+-\xi_\ell^-$, where $\xi^+_\ell= (\xi_{\ell,1}^+,\dots, \xi_{\ell,n}^+)$ and $\xi^- = (\xi_{\ell,1}^-,\dots, \xi_{\ell,n}^-)$ have non-negative components in $\mathbb Z^n$ and disjoint supports, then the 
current $[Z_B^{\gamma}]_{\rm red}$ is given by 
$$
\mathds{1}_{Z_B^\gamma} \cdot \Big[dd^c \log \Big(
\sum\limits_{\ell=1}^M \Big|\prod_{j=1}^n \zeta_j^{\xi^+_{\ell,j}} - 
\prod\limits_{j=1}^n \gamma_j^{\xi_{\ell,j}} \zeta_j^{\xi^-_{\ell,j}}\Big|^2\Big)\Big]^{n-p}
$$
by King's formula (see \cite{Dembook}, page 181).}
\end{remark}
As before let $H\,= \, H_0$ be a rational $p$-plane (passing through $0$). One can find a $\ZZ$-basis for the lattice $L_H:=H \cap \ZZ^n$, $B=\big(w_1, \dots, w_p \big)$.  
Moreover, $B$ can be completed as $D=\big(w_1,\dots ,w_p, u_1,\dots, u_{n-p} \big)$ which stands, as a set, 
as a $\mathbb{Z}$-basis for $\mathbb{Z}^n$ (if $D$ denotes the matrix of such vectors as columns, one has $\det D=\pm 1$). Note that, if $D$ and $D'$ are two such completions of $B$, one has 
$$
D' = D \cdot \begin{bmatrix} 
{\rm Id}_p & 0 \\
K & \widetilde C 
\end{bmatrix} 
$$
where $K$ and $\widetilde C$ are respectively $(n-p,p)$ and $(n-p,n-p)$ matrices 
with integer coefficients and $\det \widetilde C =(\det D)^{-1} \times \det D' = \pm 1$.  
Fix for the moment a basis $B$ and consider such a completion $D_B=D$ of $B$. 
Consider, for each $(\theta_{p+1},...,\theta_n) \in (\mathbb R/\mathbb Z)^{n-p}$, the set 
\begin{equation}\label{defDelta} 
\Delta_{H,D} (\theta) := 
\{\tau_1^{w_1} \star \dots \star \tau_p^{w_p} 
\star e^{2i\pi \theta_{p+1} u_1} \star \dots \star e^{2i\pi \theta_n u_{n-p}}\,;\, 
\tau \in (\mathbb C^*)^p\}. 
\end{equation} 
This is a $p$-dimensional analytic subset of $(\mathbb C^*)^n$ which is a toric set 
of the form $Z_B^{\gamma_u}$. 
In addition, one can parametrize $S_H := \Log^{-1} (H)$ in the following way : 
$$
S_H = \Big\{\tau_1^{w_1} 
\star \dots \star \tau_p^{w_p} 
\star e^{2\pi \theta_{p+1} u_1} \star \dots \star e^{2i\pi \theta_n u_{n-p}}~; 
\tau \in (\mathbb C^*)^p,\ (\theta_{p+1},...,\theta_n) 
\in (\mathbb R/\mathbb Z)^{n-p}\Big\}. 
$$
Therefore each $\Delta_{H,D}(\theta_{p+1}, \dots, \theta_n)$ can be considered as the fiber over 
$(\theta_{p+1},...,\theta_n)$ of the submersion $\sigma_{H,D}:$ 
\begin{gather*}\label{foliation} 
\tau_1^{w_1} \star \dots\star \tau_{p}^{w_{p}}  \star 
e^{2i\pi \theta_{p+1} u_1}\star \dots \star  e^{2i\pi \theta_n u_{n-p}}   \in S_H 
\\
\xymatrix{\ar@{|->}[d]^{\sigma_{H,D}}& \\ &}\\
(\theta_{p+1}\, , \dots ,\theta_n) \in (\RR\slash \ZZ)^{n-p}\,.
\end{gather*}
We define the positive $(p,p)$ current $T_{H,D}$  
\begin{equation}\label{definitionofthecurrent}
T_{H,D} = \int_{(\theta_{p+1},...,\theta_n)
\in (\mathbb R/\mathbb Z)^{n-p}} \big[\Delta_{H,D} (\theta_{p+1}, \dots,\theta_{n})\big] d\theta_{p+1} \dots d\theta_n\,.
\end{equation}
If one considers two completions $D= (B,U)$ and $D'=(B,U')$ of $B$, though the fibers 
$\Delta_{H,D}$ do vary 
when $D$ is changed into $D'$ (as well as the integration currents $[\Delta_{H,D}]$)), the sum $T_{H,D}$ does not since $U' = U\cdot \widetilde C$ (where $\widetilde C \in GL(n-p,\mathbb Z)$) 
and the Lebesgue measure on $(\mathbb R/\mathbb Z)^{n-p}$ is preserved 
under the action of monoidal automorphisms of the torus $(\mathbb S^1)^{n-p}$ 
whose matrix $\widetilde C$ of exponents belongs to $GL(n-p,\mathbb Z)$. 
As a result, the current $T_{H,D}$ depends only on $B$ and one can write 
$T_{H,D}= T_{H,D_B} = T_{H}^{[B]}$ for any completion $D_B$ of $B$.
On the other hand, if $U$ is fixed, it follows from 
Lemma \ref{section4-lemma1} that, if one considers 
$D=(B,U)$ and $D'=(B',U)$, where $B$ and $B'$ are two lattice basis of 
$L_H$, then $[\Delta_{H,D}(\theta)]=[\Delta_{H,D'}(\theta)]$ for 
any $\theta = (\theta_{p+1},...,\theta_n) \in (\mathbb R/\mathbb Z)^{n-p}$, hence 
$T_{H,D}= T_{H,D'}$. Accordingly, $T_{H,D_B}=T_H^{[B]}$ is in fact 
independent of $B$, and one defines in such a way a positive current 
$$
T_H = T_H^{[B]} = T_{H, \{B,U_B\}}= T_{H,D_B}
$$
which is independent of the choice of the lattice basis $B$ for $L_H = H\cap \ZZ^n$ as well as that of its 
completion $D=D_B=(B,U_B)$ as a $\mathbb Z$-basis of $\mathbb Z^n$. 
The support of $T_H$ (considered as a $(p,p)$-dimensional positive current in 
$(\mathbb C^*)^n$ is clearly $\Log^{-1}(H)=S_H$. 
\vskip 2mm
Now assume that $H_a\subset \RR^n$ is a rational affine $p$-plane obtained by translation of a rational $p$-plane $H = H_0$ via $a=(a_1,\dots, a_n) \in \RR^n\,.$ Define the linear map 
\begin{align*}
L_a: \CC^n &\rightarrow \CC^n, \\z=(z_1,\dots, z_n) &\mapsto \exp(-a)\star z = (\exp(-a_1) z_1, \dots , \exp(-a_n) z_n).
\end{align*}
Set  
$$
T_{H_a} := L_a^*(T_{H})=\int_{(\theta_{p+1},\dots, \theta_{n})\in (\RR\slash \ZZ)^{n-p}} [L_a^{-1}(\Delta_{H,D}(\theta_{p+1},\dots , \theta_{n}))]d\theta_{p+1}\dots d\theta_{n}.
$$
Accordingly, 
$$
S_{H_a} = \exp(a) \star S_{H}, 
$$
and 
$$
\Delta_{H_a, D} = \exp(a) \star \Delta_{H, D}\,.
$$
It is easily seen that the definition of $T_{H_a}$ is independent of the choice of the base point $a \in H_a\,,$ which makes us ready to propose the following definition.  
\begin{definition}\rm
Assume $\mathcal{P}$ is a weighted rational polyhedral complex of pure dimension $p$. Let $\mathcal{C}_p(\mathcal{P})$ be the family of all $p$ dimensional cells of $\mathcal{P}.$ Each $P \in \mathcal{C}_p(\mathcal{P})$ is equipped with a non-zero integral weight $m_P$ and lies in an affine $p$-plane $H_{a_P}$ which passes through a chosen base point $a_P\in P\,.$ 
Let 
$$
\mathscr T_P \, = \mathds{1}_{\Log^{-1}(\textrm{int } P)} \,\, T_{H_{a_P}} \, 
$$
be the restriction of the positive $(p,p)$-dimensional current $T_{H_{a_P}}$ 
(supported by $\Log^{-1} (H_{a_P})$)  
to $\Log^{-1}(\textrm{int } P) \subset 
\Log^{-1} (H_{a_P}) \subset (\CC^*)^n$. Here ${\rm int} (P)$ denotes the relative interior of $P$ in the affine $p$-plane 
$H_{a_P}\,.$ This definition is independent of the chosen base point $a_P$.
We define 
$$
\mathscr T_n^p (\mathcal{P}) = \sum_{P\, \in \, \mathcal{C}_p(\mathcal{P})}m_P\, \mathscr T_P~.
$$
\end{definition}

Obviously, if $\mathcal{P}$ is positively weighted, then $\mathscr T_n^p (\mathcal{P})$ is a positive current. In this article we are interested in the case where $\mathcal{P}$ is a tropical cycle $V_{\mathbb T}$. In such case, we call $\mathscr T_n^p (V_{\mathbb T})$ the \textbf{tropical current} associated to $\mathcal{P}=V_{\mathbb T}$. 
\vskip 2mm
Before stating the main theorem of this article, we introduce the following terminology.

\begin{definition}\label{subind} \rm
A set of vectors is said to be linearly sub-independent over a field  $\mathbb{K}$ if each proper subset of this set is a set of linearly independent vectors.
\end{definition}

\begin{remark}\label{remarkonsubindependency}
{\rm Suppose that the set of vectors $\{v_1,\dots,v_s \}$ is linearly sub-independent over $\mathbb{R}$ and 
there exist $a_j, b_j \in \mathbb{C}$, $j=1,\dots,s$~ such that $\sum_{j=0}^{s} a_j v_j =\sum_{j=0}^{s} b_j v_j =0.$ Then there exists a $\rho \in \mathbb{C} $ such that $a_j= \rho \, b_j$ for $j=1,\dots,s$.} 
\end{remark}

\begin{definition}\rm
A tropical $p$-cycle $V_{\mathbb{T}}\subset \RR^n$ is said to be strongly extremal if
\begin{enumerate}
 \item $V_{\mathbb{T}}$ is connected in codimension $1$~; 
 \item each $p-1$ dimensional face (facet) $W$ of $V_{\mathbb{T}}$ is adjacent to exactly $n-p+2$ polyhedra (cells) of dimension $p$~; 
 \item for each facet of $W$ of $V_{\mathbb T}$, let $\{v_1,\dots,v_{n-p+2}\}$ be the primitive vectors, one in each of the $n-p+2$ polyhedra above, that make the balancing condition hold. Then, the set of their projections along $W$, $\{h_W(v_1),\dots, h_W(v_{n-p+2})\},$ forms a sub-independent set. 
\end{enumerate}
\end{definition}

For instance, when $V_{\mathbb{T}}\subset \RR^n$ is a tropical $1$-cycle, then the strong extremality conditions 
means that the graph is $(n+1)$-valent at every vertex and the corresponding $(n+1)$- primitive vectors span $\RR^n$. It is also clear that for tropical hypersurfaces in the number $n-p+2$ is exactly $n-(n-1)+2=3$.

\begin{theorem}\label{maintheohigherdim}
If $V_{\mathbb{T}} \subset \RR^n$ is a tropical $p$-cycle, then the normal and $(p,p)$-dimensional tropical current $\mathscr T_n^p (V_{\mathbb{T}})$ is closed. If moreover $V_{\mathbb{T}}$ is strongly extremal, then $\Tnp(V_{\T})$ is strongly extremal in $\mathcal{D}'_{p,p}((\CC^*)^n)$. 
\end{theorem}

In order to make our understanding progressive, we first explore the case of 
tropical curves ($p=1$), then that of $p$-dimensional tropical cycles with a single codimension $1$ face, a {facet}. 

\subsection{Tropical $(1,1)$-dimensional currents}\label{construction11section}

In this section we study $\mathscr{T}_n^1(\Gamma)$, where $\Gamma$ is a weighted rational graph. We prove 
Theorem \ref{maintheohigherdim} in this case. Suppose an edge $e$ of $\Gamma$ of weight $m_e$ (spanning the affine line $E\subset \mathbb R^n$) is parameterized by 
$$
t\mapsto t\, v_{e}+a,
$$ 
where $\{a\}$ ($a\in \mathbb \RR^n$) is one of the vertices of $e$, $v_e=v_e^{[a\rightarrow]}\in \RR^n$ is the corresponding (inward) primitive vector for $e$ from the vertex $\{a\}$, and $t \in [0,t_0] \subset \RR$ is a real parameter, $t_0 \in [0, +\infty];$ when $t=\infty$ the edge is a ray. We complete $\{v_e\}$ to a basis $D_e$ of the lattice $\mathbb{Z}^n$, say $D_e =(v_e, U_e) =\big\{v_e,u_1^e,\dots,u_{n-1}^e\big\}$, that is, if one denotes also $D_e$ as the matrix with columns $v_e^t, (u_1^e)^t,\dots, (u_{n-1}^e)^t$, one has $\det(D_e)=\pm 1$, \textit{i.e.} $D_e \in GL(n, \mathbb{Z}).$ 
We can now define an open subset $S_{e,D_e,a}\subset S_{E}:=\Log^{-1}(E)$ as~:    
\begin{equation}\label{definitionS} 
\begin{split} 
& S_{e,D_e,a} := \Big\{\exp (a) \star \tau^{v_e} \star \exp (2i\pi \theta_2 u_1^e)  
\star \dots \star \exp(2i\pi \theta_{n} u_{n-1}^e)~;  \\
& \qquad \qquad 
\tau \in \CC^*,~ 1 < |\tau| < \exp (t_0),\ \theta =(\theta_2,...,\theta_{n})  
\in (\mathbb R/\mathbb Z)^{n-1}\Big\}\,.  
\end{split} 
\end{equation}
\vskip 1mm
\noindent
Such an open set $S_{e,D_e,a}\subset S_E$ 
(considered here as a submanifold with boundary of the manifold $S_E$ with real dimension 
$n+1$) is injectively foliated over the Cartesian product 
$\big(\RR \slash \mathbb{Z} \big)^{n-1}$ through the submersion 
\begin{gather*}\label{foliation} 
\exp (a)\star \tau^{v_e} \star  
\exp(2i\pi \theta_{2} u_1^e)\star \dots \star  \exp(2i\pi \theta_n u_{n-1}^e)   \in S_{e,D,a}  
\\
\xymatrix{\ar@{|->}[d]^{\sigma_{e,D_e,a}} & \\ &}\\
(\theta_{2}\, , \dots ,\theta_n) \in (\RR\slash \ZZ)^{n-1}\,.
\end{gather*}
One also denotes as $\boldsymbol\tau_{e,D_e,a}$ the parameterization map from 
$(\mathbb C^*)^n$ into itself which is used to get (through its inverse)  
the submersion $\sigma_{e,D_e,a}$, that is the monoidal map : 
$$
\boldsymbol \tau_{e,D_e,a} ~: 
(\tau_1,\lambda_2, \dots, \lambda_n) \in (\C^*)^n 
\mapsto \exp (a) \star \tau_1^{v_e} \star \lambda_2^{u_1^e} \star \dots \star \lambda_n^{u_{n-1}^e} 
\in (\mathbb C^*)^n. 
$$  
Denote as $\Sigma_{e,D_e,a}$ the cycle 
\begin{equation*} 
\begin{split} 
& \Sigma_{e,D_e,a}:= \partial S_{e,D_e,a}~: \\  
& (\theta_1,...,\theta_n) \in (\mathbb R/\mathbb Z)^n 
\mapsto \exp (a)\star \exp(2i\pi \theta_1\, v_e)  
\star \exp(2i\pi \theta_2\, u_1^e) \star \dots \star \exp(2i\pi \theta_n u_{n-1}^e) ~.
\end{split} 
\end{equation*} 
The support of the cycle $\Sigma_{e,D_e,a}$ equals $\Log^{-1} (\{a\})$. For each $(\theta_2,...,\theta_{n}) \in (\mathbb R/\mathbb Z)^{n-1}$, denote as $\Delta_{e,D_e,a}$ the fiber $\sigma_{e,D_e,a}^{-1} (\{(\theta_2,...,\theta_{n})\})$ of the submersion $\sigma_{e,D_e,a}$ over $(\theta_2,...,\theta_n)$ and consider 
the $(1,1)$-dimensional positive current in $(\mathbb C^*)^n$ defined as 
$$
T_{e,D_e,a} := \int_{(\theta_2,...,\theta_n)\in (\mathbb R/\mathbb Z)^{n-1}} 
[\Delta_{e,D_e,a} (\theta_2,...,\theta_n)]\, d\theta_2 \dots d\theta_n. 
$$

The current $T_{e,D_e,a}$ is obviously not closed~; but nevertheless, 
its support is the set $\Log^{-1}(e)$. As we have explained in the beginning of this section, the current $T_{e,D_e,a}$ 
is independent of the choice of the completion $D_e$ for $\{v_e\}$ 
because of the invariance of the Lebesgue measure on $(\mathbb R/\mathbb Z)^{n-1}$ 
under the action of the linear group $GL(n-1,\Z)$ (considered in the multiplicative 
sense). In fact $T_{e,D_e,a}$ depends only on 
$e$ and stands as the current $\mathscr T_e$ obtained 
as the restriction to the edge $e$ of the positive $(1,1)$-dimensional current $T_E$ (in order to check this point, one can easily reduce the situation up to translation to the case $a=0$). 
We however keep track of the averaged representation  
\begin{equation}\label{representationTp} 
\mathscr T_e = T_{e,D_e,a} := \int_{(\theta_2,...,\theta_n)\in (\mathbb R/\mathbb Z)^{n-1}} 
[\Delta_{e,D_e,a} (\theta_2,...,\theta_n)]\, d\theta_2 \dots d\theta_n, 
\end{equation} 
where the average of integration currents $[\Delta_{e,D_e,a}]$ indeed depend on 
the specified vertex $a$ of $e$ and on the completion $D_e$ of the set 
$\{v_e\}$, where $v_e=v_e^{[a\rightarrow ]}$ denotes the primitive (inward) vector spanning $E$ and emanating 
from its specified vertex $a$. \\

\begin{lemma}\label{actionofT}
Let $\omega$ be a $1$-test form on $(\CC^*)^n$, with support in a neighborhood of 
$\Log^{-1} (\{a\}) \subset (C^*)^n $, with the restriction 
$$
\omega_{|{\Log^{-1}(\{a\})}} = \sum\limits_{j=1}^n \omega_j(t_1,...,t_n)\, dt_j~.   
$$
Then  
\begin{equation}\label{stokes1}
\begin{split} 
&\langle d \mathscr T_e, \omega \rangle = \\
& \sum\limits_{j=1}^n v_{e,j} \int_{\theta \in (\mathbb \R/\mathbb \Z)^n} \omega_j \Big(
\!v_{e,1} \theta_1\! + \!\sum_{\ell =1}^{n-1} 
u_{\ell,1}^e\theta_{\ell +1},\, ...\,, 
v_{e,n} \theta_1 \!+ \!\sum\limits_{\ell =1}^{n-1} 
u_{\ell,n}^e\, \theta_{\ell +1}\Big) d\theta_1 \cdots d\theta_n\,. 
\end{split}  
\end{equation}   
\end{lemma}

\begin{proof}
By definition of differentiation of currents and Stokes' formula, it follows that, for such 
$\omega$,  
\begin{equation}
\begin{split} 
& \langle d\mathscr T_e,\omega\rangle:= 
 -\int_{(\theta_2,...,\theta_n)\in (\mathbb \R/\mathbb \Z)^{n-1}} 
\big\langle [ \Delta_{e,D_e,a}(\theta_2,...,\theta_n)], d\omega \big \rangle\, d\theta_2 \dots d\theta_n \\ 
&=  \int_{(\theta_2,...,\theta_n)\in (\mathbb \R/\mathbb \Z)^{n-1}} 
\big\langle [ \partial \Delta_{e,D_e,a}(\theta_2,...,\theta_n)], \omega \big \rangle\, d\theta_2 \dots d\theta_n, 
\end{split} 
\end{equation} 
Note that the induced orientation on boundary of each fiber $\partial \Delta_{e, D_e, a}(\theta_2,\dots, \theta_n)$ is given by $-d\theta_1\,,$ since this boundary is obtained by letting $\tau_1 =1$ in  (\ref{definitionS}). Moreover, for each fixed $(\theta_2,\dots, \theta_n)\in (\RR\slash\ZZ)^{n-1}$, $\partial \Delta_{e, D_e, a}(\theta_2,\dots, \theta_n)$ can be understood as the image 
$$
\boldsymbol{\tau}_{{e,D_e,a}}^{(\theta_2,\dots, \theta_n)}(\RR\slash \ZZ):= \boldsymbol{\tau}_{{e,D_e,a}}\big((\RR\slash \ZZ), \theta_2,\dots, \theta_n \big).
$$
Therefore, 
$$ 
\langle d\mathscr T_e,\omega\rangle = \int_{(\theta_2,...,\theta_n)\in (\mathbb \R/\mathbb \Z)^{n-1}} \int_{\theta_1 \in (\mathbb \R/\mathbb \Z)} \big(\boldsymbol{\tau}_{e,D_e,a}^{(\theta_2,\dots, \theta_n)}\big)^*(\omega)\,. 
$$
It is clear that 
\begin{equation}\label{tau-of-tj}
\big(\boldsymbol{\tau}_{e,D_e,a}^{(\theta_2,\dots, \theta_n)}\big)^*(t_j)=t_j \circ \big(\boldsymbol{\tau}_{e,D_e,a}^{(\theta_2,\dots, \theta_n)}\big)
= v_{e,j} \theta_j\! + \!\sum_{\ell =1}^{n-1} 
u_{\ell,j}^e\theta_{\ell +1},
\end{equation}
and
$$
\big(\boldsymbol{\tau}_{e,D_e,a}^{(\theta_2,\dots, \theta_n)}\big)^*(dt_j)=
d\Big(t_j \circ \big(\boldsymbol{\tau}_{e,D_e,a}^{(\theta_2,\dots, \theta_n)}\big)\Big)
= d (v_{e,j} \theta_1\! + \!\sum_{\ell =1}^{n-1} 
u_{\ell,j}^e\theta_{\ell +1})= v_{e,j}\, d\theta_1\,,
$$
which easily give the result. 
\end{proof}

\vskip 2mm

The next lemma relates the balancing condition to closedness of the corresponding currents. We refer the reader for a similar result on ``super currents" to \cite{Lagerberg}. 
Suppose every edge $e$ of $\Gamma$ is weighted by a non-zero integer $m_e$. Then, one has the following lemma.  

\begin{lemma}\label{closed-balanced}
Let $\mathcal P$ a weighted rational $1$-polyedral complex in $\mathbb R^n$, 
$\{a\}$ be one of its vertices and 
$\omega$ be a $1$-test form in $(\mathbb C^*)^n$ supported in an open neighborhood of 
$\Log^{-1}(\{a\})$. One has 
\begin{equation}
\langle d\mathscr T_n^1(\mathcal P),\omega \rangle = 
\sum_{\{e\in \mathcal C_1(\mathcal P)\,;\, \{a\}\prec e\}} m_e\, \langle d\mathscr T_e\,, \omega \rangle = 0\quad 
\Longleftrightarrow \ 
\sum_{\{e\in \mathcal C_1(\mathcal P)\,;\, \{a\}\prec e\}} m_e v_{e}^{[a\rightarrow]} =0\,,
\end{equation}
where $\{a\} \prec e$ means that $\{a\}$ is a vertex of the edge $e$ and $v_{e}^{[a\rightarrow]}$ denotes then the inward primitive vector contained in the edge $e$ and pointing away from $a$ ; In particular, the tropical current $\mathscr T_n^1(V_\mathbb T)$ attached to a tropical curve 
$V_\mathbb T$ is closed. 
\end{lemma}

\begin{proof}
To prove the lemma it is enough to check the result for any $1$-test form $\omega$ 
in a neighborhood of $\Log^{-1} (\{a\})$ in $(\mathbb \C^*)^n$ such that 
$\omega = e^{2i\pi \langle \nu,\theta\rangle}\, d\theta_j$ for 
some $j\in \{1,...,n\}$ and $\nu \in \mathbb \Z^n$. This follows 
from the fact that the characters $\theta \mapsto 
\chi_{n,\nu} (\theta):=e^{2i\pi \langle \nu,\theta\rangle}$ ($\nu \in \mathbb \Z^n$) 
form an orthonormal basis for the Hilbert space $L^2_{\mathbb \C}((\mathbb R/\mathbb \Z)^n, d\theta)$.   
Then the equivalence stated here follows from the formula 
\eqref{stokes1} established in Lemma \ref{actionofT}. 
The second claim follows from the fact that 
the balancing condition is fulfilled 
at any vertex $\{a\}$ of any tropical curve $V_{\mathbb \T}$. 
\end{proof}

Recall that for a tropical curve $\Gamma \subset \RR^n$ strong extremality means $(n+1)$-valency 
for any vertex $\{a\}$ and sub-independency of the set whose elements are the $(n+1)$ primitive vectors 
$v_e^{[a\rightarrow]}$ ($e\in \mathcal C_1(\Gamma)$ such that $\{a\} \prec e$).

\begin{theorem}\label{maintheorem}
Let $\Gamma \subset \RR^n$ be a strongly extremal tropical curve. Then the 
$(1,1)$-dimensional closed current normal $\mathscr {T}_n^1(\Gamma)$ is strongly extremal in $\mathcal{D}_{1,1}'(\CC^*)^n.$
\end{theorem}

We first prove first Theorem \ref{maintheorem} for a tropical curve $\Gamma$ which has only one vertex. 

\begin{lemma}\label{mainlemma}
Suppose $\Gamma \in \RR^n$ is a strongly extremal tropical curve with only one vertex at the origin.  Then $\mathscr{T}_n^1(\Gamma) $ is strongly extremal in $\mathcal{D}_{1,1}'(\CC^*)^n\,.$
\end{lemma}

\begin{proof}
The proof of the lemma is divided into three steps. \\

Each edge $e \in \mathcal C_1(\Gamma)$ is contained in an affine line $E$. For such a $E$ consider $w=v_e$ the inward primitive vector $w=v_e^{[0\rightarrow]}$ initiated 
from the vertex $\{0\}$, lying in $E$. We fix an arbitrary completion $D_e$ of 
$\{v_e\}$ with vectors $u^e_1,...,u_{n-1}^e$ in $\Z^n$.  One has 
$$
\mathscr T_{1}^n (\Gamma) = \sum\limits_{e\in \mathcal C_1(\Gamma)} m_e\, T_{e,D_e,\{0\}} 
$$
as seen in Subsection \ref{construction11section} above. 
\vskip 2mm
\noindent
We assume from now on that $\widetilde{\mathscr T}$ is a $(1,1)$-dimensional normal closed current in 
$(\mathbb C^*)^n$ with support equal to that of 
$\mathscr T_1^n (\Gamma)$, {\it i.e.} ${\rm Supp}\, 
(\widetilde{\mathscr T}) = \Log^{-1} (\Gamma)$. 
\vskip 2mm 
\noindent
\textbf{Step 1.} For any $e\in \mathcal C_1(\Gamma)$, let 
$\mathcal U_e$ be the open subset of $(\mathbb C^*)^n$ defined as 
$$
\mathcal U_e := \Log^{-1} \big(\R^n \setminus \bigcup_{\stackrel{e'\in \mathcal C_1(\Gamma)}
{e'\not=e}} |e'|\big). 
$$
It follows from Theorem \ref{supportdecom} that, for each $e\in \mathcal C_1(\Gamma)$, there is  
a Radon measure $d\mu_e$ 
on $(\RR \slash \mathbb{Z})^{n-1}$ such that (as currents in the open subset $\mathcal U_{e}$ 
of $(\mathbb C^*)^n$) : 
$$
\widetilde {\mathscr T}_{|\mathcal U_e} = 
\int\limits_{(\theta_2,...,\theta_n)\in (\mathbb \R/\mathbb \Z)^{n-1}} 
[\Delta_{e,D_e,0} (\theta_2,...,\theta_n)]\, d\mu_e (\theta_2,...,\theta_n). 
$$ 
Since the normal current $\widetilde {\mathscr T}_{|\mathcal U_e}$ extends globally as the $(1,1)$-dimensional normal closed current $\widetilde {\mathscr T}$  
in the whole ambient manifold $(\mathbb C^*)^n$, one can certainly define 
$(1,1)$-dimensional normal current $\widetilde {\mathscr T}_e$ in 
$(\mathbb \C^*)^n$ as 
\begin{equation}\label{expressiontildeT} 
\widetilde {\mathscr T}_e := \int\limits_{(\theta_2,...,\theta_n)\in (\mathbb \R/\mathbb \Z)^{n-1}} 
[\Delta_{e,D_e,0} (\theta_2,...,\theta_n)]\, d\mu_e (\theta_2,...,\theta_n).  
\end{equation} 
The support of $\widetilde {\mathscr T}_e$ equals $\Log^{-1}(e)$, which implies that 
all currents $\widetilde{\mathscr T}_{e'}$ (for $e'\in \mathcal C_1(\Gamma)$ such that 
$e'$ is distinct from $e$) vanish in $U_e$. Hence 
$\widetilde {\mathscr T} = \sum_{e\in \mathcal C_1(\Gamma)} 
\widetilde {\mathscr T}_e$ in each $U_e$. Hence the current 
$\widetilde {\mathscr T} - \sum_{e\in \mathcal C_1(\Gamma)} \widetilde {\mathscr T}_e$ 
(which is normal) is supported by $\Log^{-1} (\{0\})$ 
which equals to the real $n$-dimensional torus and therefore has Cauchy-Riemann dimension $0$. 
It follows then from Theorem \ref{supportsmall} that one has the decomposition : 
$$
\widetilde {\mathscr T}  = \sum\limits_{e\in \mathcal C_1(\Gamma)} \widetilde {\mathscr T}_e 
$$
(as currents this time in the whole ambient space $(\mathbb \C^*)^n$).  

\begin{remark}\label{fixinglatticebases}
\rm{
Although the current $\mathscr{T}_n^{1}(\Gamma)$ is not dependent on completions of $v_e$ to lattice bases $D_e$, the representation in (\ref{expressiontildeT}) is. The representation, indeed depends on the chosen foliation which comes from the completions $D_e$ of $v_e$ to a $\ZZ$-basis for every edge $e$ of $\Gamma$. Therefore as mentioned before, at this point we need to fix a lattice basis for each of the edges of the tropical graph. 
}
\end{remark}

\vskip 2mm
\noindent 
\textbf{Step 2.} One can repeat the proof of Lemma \ref{actionofT} for each edge $e\in \mathcal C_1(\Gamma)$ and use the expression \eqref{expressiontildeT} of $\widetilde {\mathscr T}_e$, in order to get the following result.  

\begin{lemma}\label{actionofT2}
Let $\omega$ be a $1$-test form on $(\CC^*)^n$, with support in a neighborhood of 
$\Log^{-1} (\{0\})$ with restriction given by

$$
\omega_{|\Log^{-1} (\{0\})} = 
\sum\limits_{j=1}^n \omega_j(t_1,...,t_n)\, dt_j.   
$$
Then  
\begin{equation}\label{stokes2} 
\begin{split} 
& \langle d \widetilde {\mathscr T_e}, \omega \rangle = \sum\limits_{j=1}^n v_{e,j}\times \\
&  \int_{\theta \in (\mathbb \R/\mathbb \Z)^n} 
\omega_j \Big(
v_{e,j} \theta_1 + \sum\limits_{\ell =1}^{n-1} 
u_{\ell,1}^e\theta_{\ell +1},\, ...\,, 
v_{e,n} \theta_1 + \sum\limits_{\ell =1}^{n-1} 
u_{\ell,n}^e\, \theta_{\ell +1}\Big)\, d\theta_1\otimes d\mu_e (\theta_2,...,\theta_n)\,. 
\end{split}  
\end{equation}   
\end{lemma}
\vskip 2mm
\noindent
\textbf{Step 3.} The current $\widetilde{\mathscr T} =\sum_{e\in \mathcal C_1(\Gamma)} 
\widetilde{\mathscr T}_e$ 
is closed by hypothesis. We try to fully exploit this property in order to derive  
information on the measures $\mu_e$, $e\in \mathcal C_1(\Gamma)$. 
To do that, we use the fact that a Radon measure $d\mu$ on the group 
$(\mathbb R/\mathbb \Z)^{n-1}$ is characterized by the complete list of its Fourier 
coefficients 
\begin{equation*}
\begin{split} 
& \widehat \mu (\nu) = \int_{[0,1]^{n-1}} \chi_{n-1,\nu} (\theta_2,...,\theta_n)\, d\theta_2\dots d\theta_n : 
= \int_{[0,1]^{n-1}} \exp (-i \langle \nu,\theta\rangle)\, d\theta_2\dots d\theta_n \\
& \qquad  (\nu \in \mathbb \Z^{n-1}).  
\end{split} 
\end{equation*}  
Fix $e\in \mathcal C_1(\Gamma)$. 
Let $\omega_\nu^{[1]}$ be a $1$-test form on $(\CC^*)^n$, with support in a neighborhood of 
$\Log^{-1} (\{0\})$ such that its restriction is given by 
$$
(\omega_\nu^{[1]})_{|\Log^{-1} (\{0\})} = \chi_{n,\nu} (t_1,...,t_n)\, dt_1.    
$$
After simplifications, \eqref{stokes2} reduces to the scalar equation : 
\begin{equation} 
\langle d\, \widetilde{\mathscr T}_e , \omega_\nu^{[1]} \rangle 
= \delta^0_{\langle \nu, v_e \rangle} \, \widehat \mu_e
\big(-\langle \nu, u_1^e \rangle, \dots, -\langle \nu, u_{n-1}^e \rangle\big)\, v_{e,1}
\end{equation} 
($\delta_\alpha^\eta$ denotes here the Kronecker's symbol).  
Since $\widetilde{\mathscr T}$ is closed we conclude, after performing the 
same computations for all $e$ in $\mathcal C_1(\Gamma)$, that  
\begin{equation}\label{scalarequation11}
0= \langle d\, \widetilde{\mathscr T}, \omega_\nu^{[1]} 
\rangle = \sum_{e\in \mathcal C_1(\Gamma)} \langle d\, \widetilde{\mathscr T}_e, \omega_\nu^{[1]}\rangle 
= \sum_{e\in \mathcal C_1(\Gamma)}  
\delta^0_{\langle \nu, v_{e} \rangle}\, \widehat \mu_e
\big(-\langle \nu, u_1^e \rangle, \dots, -\langle \nu, u_{n-1}^e \rangle\big)\, v_{e,1}\,. 
\end{equation} 
If one performs the same operations when $\omega_\nu^{[1]}$ is replaced by 
$\omega_\nu^{[j]}$ ($1\leq j\leq n$) such that 
$$(\omega_\nu^{[j]})_{|\Log^{-1} (\{0\})} = \chi_{n,\nu} (t_1,...,t_n)\, dt_j,$$
one gets the vectorial equation 
\begin{equation}\label{vectorial}
\sum_{e\in \mathcal C^1(\Gamma)}
\ \delta^0_{\langle \nu, v_{e} \rangle} 
\widehat{\mu}_e \big(-\langle \nu, u_1^e \rangle, \dots, -\langle \nu, u_{n-1}^e \rangle\big)\, v_e =0.
\end{equation}
Equation \eqref{vectorial} implies the two following facts :  
\begin{itemize} 
\item Taking $\nu = (0,\dots,0)$ leads to 
$$
\sum_{e\in \mathcal C_1(\Gamma)} \widehat\mu_e(0, \dots, 0)\, v_e =0. 
$$
Recall that the balancing condition 
$\sum_{e\in \mathcal C^1(\Gamma)} m_e\, v_e=0$
is also satisfied, it follows from the sub-independency hypothesis  
(see Remark \ref{remarkonsubindependency}) 
that there exists a complex number $\rho$ such that 
$$
\widehat \mu_e(0,\dots,0) =\rho \, m_e\quad \forall\, e\in \mathcal C_1(\Gamma).  
$$

\item Let ${\ell}=(\ell_2,\dots,\ell_{n})\neq (0,\dots,0)$ be an arbitrary non-zero integral vector. Fix $e\in \mathcal C_1(\Gamma)$. There exists a unique $\nu_e\in \mathbb \Z^n$ such that at the same time  
$\langle \nu_e, v_e \rangle=~0$ and $\langle \nu_e, u_j^e\rangle=- \ell_{j+1}$ 
for $j=1,...,n-1$, since $D_e=\{v_e,u_1^e,...,u_{n-1}^e\}$ is a $\ZZ$-basis of $\ZZ^n\,.$ Since the graph $\Gamma$ is $(n+1)$-valent 
and the $(n+1)$-primitive vectors $v_{e'}$ ($e'\in \mathcal C_1(\Gamma)$) 
affinely span the whole $\RR^n$, there exists at least one edge $e'[e]$ (distinct from $e$) of 
$\Gamma$ such that $\langle \nu_e,v_{e'[e]}\rangle \not=0$, thus $\delta^0_{\langle \nu_e,v_{e'[e]}\rangle} =0$.   
Therefore, in view of Remark \ref{remarkonsubindependency}, all of the 
coefficients involved in the vectorial equation (\ref{vectorial}) must vanish, as well as 
$\delta^0_{\langle \nu_e, v_e\rangle}\widehat{\mu}_{e}(\ell)=\widehat \mu_e(\ell)$, that is $\widehat \mu_e (\ell) =0$. 
Consequently, for every $0\not=\ell\in \mathbb Z^{n-1},$ we have $\widehat{\mu}_{e}(\ell)=0$. 
\end{itemize}

It means that every $d\mu_e$ ($e\in \mathcal C_1(\Gamma)$) 
is a Lebesgue measure given by $d\mu_e(\theta_2,...,\theta_n) = \rho\, m_e\, d\theta_2 \dots \theta_n$, and therefore,  
$\widetilde {\mathscr T} = \rho \mathscr T_n^1(\Gamma)$. 
This concludes to the strong extremality of $\mathscr T_n^1(\Gamma)$. 
\end{proof}

\vskip 2mm
Now it is easy to prove the Theorem \ref{maintheorem}. 

\begin{proof}[Proof of Theorem \ref{maintheorem}.]
Let $\widetilde{\mathscr T}$ be a closed $(1,1)$-dimensional 
normal current with support $\Log^{-1}(|\Gamma|)$. For any vertex $a$ of $\Gamma$ there is an open neighborhood $\mathcal V_a$ of $a$ in $\RR^n$ which does not contain any other vertex of the tropical curve $\Gamma$. 
We are thus reduced to the situation of a tropical curve with just one vertex. It follows then from the Lemma \ref{mainlemma} (the reasoning may be applied locally, in the open set $\Log^{-1} (\mathcal V_a)$ instead as in 
$\Log^{-1}(\R^n) = (\mathbb  C^*)^n$)  that for each vertex $\{a\}$ of $\Gamma$ there is a complex number $\rho_a$ 
such that
$$
\widetilde{\mathscr T}_{|_{\Log^{-1}(\mathcal V_a)}}= \rho_a \, {\mathscr T_n^1(\Gamma)}_{|_{\Log^{-1}(\mathcal V_a)}}. 
$$ 
Similarly for an adjacent vertex $\{b\}$, we can write for some complex number $\rho_b$ 
$$
\widetilde{\mathscr T}_{|_{\Log^{-1}(\mathcal V_b)}}= \rho_b \, {\mathscr T_n^1(\Gamma)}_{|_{\Log^{-1}(\mathcal V_b)}}. 
$$ 
On the other hand, if $\{a\}$ and $\{b\}$ are connected via the edge $e$, 
then we have, using the notations from the previous lemma, that 
$$
\widetilde{\mathscr T}_e = \rho_a \, m_e \, \mathscr T_e = 
\rho_b \, m_e \, \mathscr T_e 
$$
(as currents in some open neighborhood of $\Log^{-1}(e\setminus \{a,b\})$ in 
$(\CC^*)^n$). Hence $\rho_a=\rho_b$. Since $\Gamma$ is strongly extremal and thus connected, one can show 
that all $\rho_a$ are indeed equal by taking a chain of successive adjacent vertices 
from $\{a\}$ to an arbitrary other vertex.   
\end{proof}


\subsection{Tropical $(p,p)$-dimensional currents}

In this section we prove the Theorem \ref{maintheohigherdim}. We start by treating the simplest case, namely when  $V_{\mathbb{T}} \subset \RR^n$ is a tropical $p$-cycle 
with only one facet $W$. Note that such a hypothesis implies that this facet is in fact an affine $(p-1)$-plane 
in $\mathbb R^n$ and that all $p$-cells are of the form $\mathbb [0,\infty[\times v_P + W$ for some primitive inward vector $v_P=v_P^{[W\rightarrow ]}$. Let us analyze the current $\mathscr T_n^p (V_{\mathbb{T}})$ in that particular case. Assume that $W$ (which is here assumed 
to be the sole facet of $V_{\mathbb T}$) passes through the origin and is the common facet of the $p$-dimensional polyhedra $P_1,\dots,P_s\,$, $s\geq 3$, with corresponding weights $m_P$.  We have already shown in the beginning of Section \ref{tropicalcurrents} 
that in the definition of $\mathscr T_n^p (V_\T)$ 
$$
\mathscr T_n^p (V_\T) = \sum\limits_{P \in \{P_1,\dots,P_s\}} m_P\, 
\mathscr T_P,  
$$
is independent of the choice of the base point, and $\ZZ$-bases for  $P \cap \ZZ^n$ as well as their completions to $\ZZ$-bases of $\ZZ^n\,.$ Accordingly, we choose $\{w_1,\dots, w_{p-1}\}$ a $\ZZ$-basis for $W \cap \ZZ^n$ and for each $P\in \{P_1,...,P_s\}$, we choose the inward primitive vector $v_P=v_P^{[W\rightarrow]}\in \Z^n$ pointing inward 
$P$ from the origin such that $\{w_1,\dots, w_{p-1}, v_P\}$ is a $\ZZ$-basis for $H_P \cap \ZZ^n$ where $H_P$ is the $p$-plane containing $P\,.$ Also, the balancing condition means (see Remark \ref{remark-higher-balancing}) that
every $p\times p$ minor of the $n \times p$ matrix of columns $\big(w_1,\dots,w_{p-1}, \sum_{P} m_P v_P \big)$ vanishes. Equivalently, this implies that under the projection along $W$, $h_W: \RR^n \rightarrow \RR^{n-p+1}$, we have 
\begin{equation}\label{balancing-with-proj}
\sum_{P} \, m_p\, h_{W}(v_P) =0. 
\end{equation}
\noindent
Furthermore, we extend each $\{w_1, \dots, w_{p-1}, v_P\}$ to 
$$D_P =\{w_1,\dots,w_{p-1}, v_P, u_1^P,\dots, u_{n-p}^P\}\,,$$
 a $\ZZ$-basis of $\ZZ^n$. 
Correspondingly, we define the open subset $S_{D_P}$ of $\Log^{-1}(H_P)$ by
\begin{equation}
\begin{split} 
& S_{D_P} = 
\big\{ ( \tau_1^{w_1} \star \dots\star \tau_{p-1}^{w_{p-1}} \star \tau^{v_P} \star 
\exp(2i\pi \theta_{p+1} u^P_1) \star \dots \star \exp(2i\pi \theta_n u^P_{n-p})~; \\
& (\tau_1,...,\tau_{p-1}, \tau)\in (\mathbb C^*)^{p},\  |\tau|>1,\ (\theta_{p+1},...,\theta_n) \in 
(\mathbb R/\mathbb Z)^{n-p}\big\}\,.  
\end{split} 
\end{equation}
Each $S_{D_P}$ (which is a $(n+p)$-dimensional real manifold) 
is injectively foliated over $(\mathbb R/\mathbb \Z)^{n-p}$ through the submersion 
\begin{gather*}\label{foliation} 
 \tau_1^{w_1} \star \dots 
\star \tau^{w_{p-1}} \star \tau^{v_P} \star 
\exp(2i\pi \theta_{p+1} u^P_1) \star \dots \star \exp(2i\pi \theta_n u^P_{n-p}) \in 
S_{D_P} 
\\
\xymatrix{\ar@{|->}[d]^{\sigma_{P}} & \\ &}\\
(\theta_{p+1}\, , \dots ,\theta_n) \in (\RR\slash \ZZ)^{n-p}\,.
\end{gather*}
We again denote the fiber over $(\theta_{p+1},\dots,\theta_n) \in (\mathbb{R}\slash \mathbb{Z})^{n-p}$, 
as $\Delta_{P}(\theta_{p+1},\dots,\theta_n)$. 
Note that the complex dimension of each $\Delta_{P}(\theta_{p+1},\dots,\theta_n)$ 
is $p$ and that the boundary of such a fiber is a real analytic $(2p-1)$-cycle. 
Denote by $\boldsymbol\tau_{P}$ the parameterization map 
\begin{align*}
&\boldsymbol \tau_{P} ~: 
(\tau, \theta_{p+1},\dots,\theta_{n}) \in (\mathbb C^*)^{p}\times (\RR\slash \ZZ)^{n-p}  
\mapsto  \\ & \tau_1^{w_1} 
\star \dots \tau_{p-1}^{w_{p-1}} \star 
\tau_p^{v_P} \star \exp(2i\pi \theta_{p+1} u^P_1) \star \dots \star \exp(2i\pi \theta_n u^P_{n-p})
\in (\mathbb C^*)^{n}. 
\end{align*}  
Identifying $\CC^{p-1}\times (\mathbb{R}\slash \mathbb{Z}) \simeq (\RR^+)^{p-1}\times (\mathbb{R}\slash \mathbb{Z})^{p}$, $\partial \Delta_{P}(\theta_{p+1},\dots,\theta_n)$ (with orientation induced by $-d\theta_p$) can be therefore understood as the image
\begin{multline}\label{tau_P}
\boldsymbol \tau_{P} \big((\RR^+)^{p-1}\times (\RR\slash \ZZ)^p \times \{(\theta_{p+1}, \dots, \theta_{n})\}\big)
\\=: \boldsymbol \tau_{P}^{(\theta_{p+1}, \dots, \theta_{n})} \big((\RR^+)^{p-1}\times (\RR\slash \ZZ)^p\big) .
\end{multline}
By definition of the tropical current associated to $V_{\mathbb T}$,
\begin{multline*}
\mathscr T_n^p (V_{\mathbb T}) =\!\! \sum\limits_{P\in C_p(V_\mathbb T)}\!\!\! m_P\,
 \mathscr T_{P} = \!\!\!\!\! \sum\limits_{P\in C_p(V_\mathbb T)}\!\!\! \!m_P\,\! \int_{(\theta_{p+1},...,\theta_n) 
\in (\mathbb R/\mathbb \Z)^{n-p}} \!\!\!\!\!
[\Delta_{P}(\theta_{p+1},...,\theta_n)]\, d\theta_{p+1} \dots d\theta_n\, .  
\end{multline*}

\begin{lemma}\label{mainlem-higherdim} 
Let $V_{\mathbb T}\subset \mathbb R^n$ be a $p$-tropical cycle such that 
$\mathcal C_{p-1}(V_\mathbb T) =\{W\}\ni \{0\}$. The current $\mathscr T_n^p(V_\mathbb T)$, which can then be decomposed as 
$$
\mathscr T_n^p (V_{\mathbb T}) = 
\sum\limits_{P\in C_p(V_\mathbb T)} m_P\, \int_{(\theta_{p+1},...,\theta_n) 
\in (\mathbb R/\mathbb \Z)^{n-p}} 
[\Delta_{P}(\theta_{p+1},...,\theta_n)]\, d\theta_{p+1} \dots d\theta_n
$$
is a closed $(p,p)$-dimensional current. Moreover, assume ${\rm card} (C_p(V_{\mathbb T})) = n-p+2$, with inward vectors $v_P$ in $P$, that make the balancing condition hold. If the set of vectors $\{h_W(v_P)\,;\, P\in C_p(V_{\mathbb T})\}$, $h_W$ being the projection along $W$, is linearly sub-independent, then $\mathscr T_n^p (V_{\mathbb T})$ is strongly extremal.  
\end{lemma}
\begin{proof}
The proof is similar to that of Lemmas \ref{closed-balanced} and \ref{mainlemma}. 
Consider a $(2p-1)$-test form $\omega_{\eta,\nu}^{[K,J]}$ in $\CC^*$ such that in a neighborhood of the $\Log^{-1}(0)\subset \Log^{-1}(W)$ is expressed in polar coordinates $z_j = r_j\, e^{2i\pi t_j}$, $j=1,...,n$, by 
\begin{equation*}
\begin{split} 
& \omega_{\eta,\nu}^{[K,J]}(r_1,...,r_{n},t_1,\dots, t_n)  
  = \, \eta (r_1,...,r_{n}) \, \chi_{n,\nu}(t_1,\dots,t_n) 
\bigwedge_{k\in K} dr_k  \wedge 
\bigwedge_{j\in J} dt_j,  
\end{split} 
\end{equation*} 
where $ K\subset J\subset \{1,...,n\}$ with $|K|=p-1\,, |J|= p\,.$ $\eta$ a test function in $r=(r_1,\dots,r_n)$. Also $\nu\in \Z^n$ and $\chi_{n,\nu}$ denotes as before the character $t=(t_1,\dots, t_n) \mapsto \exp(2i\pi \langle \nu,t\rangle)$ on 
the torus $(\mathbb R/\mathbb Z)^n$. Thanks to Fourier analysis, looking at the application of $\mathscr{T}_n^p(V_{\T})$ on these forms one can extract all information needed in order to verify closedness as well as extremality of $\mathscr{T}_n^p(V_{\T})$.

\vskip 2mm
By definition of the exterior derivative of a current and the Stokes' formula, taking into account the orientation induced on the boundary of each $\Delta_P(\theta_{p+1},\dots,\theta_{n})\,,$
\begin{equation}\label{comput1}  
\begin{split} 
& \big \langle d \mathscr T_n^p (V_{\mathbb T}),
\omega_{\eta,{\nu}}^{[K,J]} \big \rangle = 
\sum\limits_{P 
\in \mathcal C_p(V_{\mathbb T})}m_P \big \langle  \,d \mathscr T_P ,
\omega_{\eta,{\nu}}^{[K,J]} \big \rangle = \\ &
\int_{(\theta_{p+1},\dots, \theta_{n})\in( \RR\slash \ZZ)^{n-p}}
\sum\limits_{P 
\in \mathcal C_p(V_{\mathbb T})} 
m_P \big \langle \partial \Delta_P(\theta_{p+1},\dots, \theta_{n}), \omega_{\eta,{\nu}}^{[K,J]} \big \rangle\,.
\end{split} 
\end{equation} 
Therefore by (\ref{tau_P}), 
\begin{equation}\label{comput1-cont}
\begin{split}
& \big \langle d \mathscr T_n^p (V_{\mathbb T}), 
\omega_{\eta,{\nu}}^{[K,J]} \big \rangle = \\
&\int_{(\theta_{p+1},\dots, \theta_{n})\in( \RR\slash \ZZ)^{n-p}} 
 \sum\limits_{P 
\in \mathcal C_p(V_{\mathbb T})} 
m_P 
\Big(\int_{(\mathbb R^+)^{p-1}\times (\RR\slash\ZZ)^{p}} (\boldsymbol{\tau}_{P}^{(\theta_{p+1}, \dots, \theta_{n})})^* \big(\eta (r ) \, dr_K \wedge \chi_{n,\nu} (t) \, dt_J \big)\Big)\,.
\end{split} 
\end{equation} 
\noindent
A computation similar to Lemma \ref{actionofT} and (\ref{scalarequation11}) gives the scalar equation
\begin{equation}\label{comput-result}  
\begin{split} 
& \big \langle d 
\mathscr T_n^p (V_{\mathbb T}),
\omega_{\eta,{\nu}}^{[K,J]} \big \rangle = \sum\limits_{P 
\in \mathcal C_p(V_{\mathbb T})} 
m_P  \, \times\, M_{\eta, K, W } \times \\  & \times \, \Big( 
\prod_{\ell=1}^{p-1} \delta^0 _{\langle \nu,w_\ell\rangle}\Big) 
\delta^0_{\langle \nu,v_P\rangle} \, 
\delta^0_{\langle \nu,u_1^P\rangle}\, \cdots \, 
\delta^0_{\langle \nu,u_{n-p}^P\rangle}\, \textrm{Det}_J \big(w_1,\dots, w_{p-1}, v_P \big) .  
\end{split} 
\end{equation} 
In the above $M_{\eta, K, W }$ is a constant coming from integration of $\big(\boldsymbol{\tau}_{P}^{(\theta_{p+1}, \dots, \theta_{n})})^* \big(\eta (r ) \, dr_K\big)$ depending on $\eta, K, W\,$ which can be chosen to be $1$\,. $\textrm{Det}_J \big(w_1,\dots, w_{p-1}, v_P \big)$ denotes the $p\times p$ minor of the $n \times p$ matrix $\big(w_1,\dots, w_{p-1}, v_P \big)$ corresponding to the rows with indices $j\in J$, this term appears from $\big(\boldsymbol{\tau}_{P}^{(\theta_{p+1}, \dots, \theta_{n})})^*\big(dt_J\big)$ in (\ref{comput1-cont})  (compare to (\ref{tau-of-tj})). For any $0 \not=\, \nu \in \ZZ^n$, (\ref{comput-result}) becomes zero. Assuming $\nu = 0~,$ yields
\begin{multline*}
\big \langle d \mathscr T_n^p (V_{\mathbb T}), \omega_{\eta,{\nu}}^{[K,J]} \big \rangle =0
 \quad \text{ if and only if } \\ 
\qquad \qquad \qquad \qquad\qquad\qquad  \textrm{Det}_J\big(w_1,\dots, w_{p-1}, \sum_{P 
\in \mathcal C_p(V_{\mathbb T})} m_P\, v_P \big) =0, \,\\ \forall \, J \subset \{1,\dots,n\}, |J|=p\,. 
\end{multline*}
Which implies the equivalence of $d$-closedness of $\mathscr T_n^p (V_{\mathbb T})$ and the balancing condition. 
\vskip 2mm
For any $P\in \mathcal C_p(V_{\mathbb T})$, let $\mathcal U_P$ be the open subset of 
$(\mathbb C^*)^n$ defined 
as 
$$
\mathcal U_{P} := \Log^{-1} \Big( 
\mathbb R^n \setminus \bigcup\limits_{\stackrel{P'\in \mathcal C_p(V_{\mathbb T})}
{P'\not=P}} |P'|\Big)
$$
Suppose that $\widetilde {\mathscr T}$ is a $(p,p)$-dimensional normal current in $(\CC^*)^n$ 
with support exactly $\Log^{-1} (V_{\mathbb T})$. 
As in the $1$-dimensional case, by Theorem \ref{supportdecom}, for any $P\in \mathcal C_p(V_{\mathbb T})$, there exists a unique Radon measure 
$d\mu_P$  
on $(\mathbb R/\mathbb Z)^{n-p}$ such as (as currents in the open subset 
$\mathcal U_P \subset (\mathbb C^*)^n$) one has 
$$
\widetilde {\mathscr T}_{|\mathcal U_P} 
= \int_{(\theta_{p+1},...,\theta_n) \in (\mathbb R/\mathbb Z)^{n-p}} 
[\Delta_{P,D_P,W} (\theta_{p+1},...,\theta_n)] \, d\mu_P (\theta_{p+1},...,\theta_n). 
$$
Since $\widetilde {\mathscr T}_{|\mathcal U_P}$ extends globally as 
the normal closed current to the 
whole of $(\mathbb C^*)^n$, one defines normal currents $\widetilde {\mathscr T}_P$ 
on $(\C^*)^n$ by setting 
$$
\widetilde {\mathscr T}_P := \int_{(\theta_{p+1},...,\theta_n) \in (\mathbb R/\mathbb Z)^{n-p}} 
[\Delta_{P,D_P,W} (\theta_{p+1},...,\theta_n)] \, d\mu_P (\theta_{p+1},...,\theta_n). 
$$
The normal $(p,p)$-dimensional current $\widetilde {\mathscr T} - \sum_{P\in \mathcal C_p(V_{\mathbb T})} 
\widetilde {\mathscr T}_P$, which is supported by $\Log^{-1} (W)$, equals zero for dimension reasons 
thanks 
to theorem $\ref{supportsmall}$, so that one has (as currents in $(\C^*)^n$ this time) the 
representation (which indeed depends on the chosen foliation):
$$
\widetilde {\mathscr T} = 
\sum\limits_{P \in \mathcal C_p(V_{\mathbb T})} 
\int_{(\theta_{p+1},...,\theta_n) \in (\mathbb R/\mathbb Z)^{n-p}} 
[\Delta_{P,D_P,W} (\theta_{p+1},...,\theta_n)] \, d\mu_P (\theta_{p+1},...,\theta_n). 
$$
Using the fact that $\langle d\widetilde {\mathscr T},\omega_{\eta,\nu}^{[K,J]}\rangle =0$ for any 
$\nu\in \mathbb Z^n$, any $K\subset J \subset \{1,\dots,n\}$ with $|J|=|K|+1= p\,,$ any test function $\eta$ in $r$ with non-zero integral 
leads to 
\begin{multline}\label{comput2}  
\sum\limits_{P \in \mathcal C_p(V_{\mathbb T})} 
\Big( 
\prod_{\ell=1}^{p-1} \delta^0 _{\langle \nu,w_\ell\rangle}\Big) 
\delta^0_{\langle \nu,v_P\rangle}\, \widehat \mu_P \big( 
- \langle \nu,u_1^P\rangle,...,-\langle \nu, u_{n-p}^P\rangle\big)\, \text{Det}_J (w_1,...,w_{p-1}, v_P) = 0\,, \\ \quad  J \subset \{1,\dots,n\},\,  |J|=p\,. 
\end{multline}
Recall that by hypothesis the set $\{h_W(v_P)\,;\, P \in \mathcal C_p(V_{\mathbb T})\}$ is linearly sub-independent and spans $W^{\perp}$ as an $\RR$-basis, where $h_W$ is the projection along $W$. The balancing condition also gives $\sum_{P \in \mathcal C_p(V_\mathbb T)}m_P h_W(v_P) =0\,.$ 
\vskip 2mm
\noindent
Similar to the bottom of the proof of Lemma  \ref{mainlemma} we deduce~: 
\begin{itemize}
  \item Choosing $\nu=0$, together with sub-independency implies that there exists a complex number $\rho$ such that $\widehat \mu_P (0,...,0) =\rho\, m_P$ by the Remarks \ref{remarkonsubindependency} and \ref{balancing-with-proj} for every $P$.
  \item Now assume $(\ell_{p+1},...,\ell_n) \in \Z^{n-p} \,$ is any non-zero vector. Since for any $P,$ $\{w_1,...,w_{p-1},v_P,u_1^P,...,u_{n-p}^P\}$ is a lattice basis of $\Z^n$, there exists a unique $\nu_P\in \Z^n\cap W^{\perp}$ such that at the same time $\langle \nu_P,v_P\rangle =0,$ $\langle \nu_P,u_j^P\rangle =- \ell_{p+j}$ for $j=1,...,n-p\,$. However, for at least one $P' \not= P,$ $\langle \nu_P , v_{P'}\rangle \not= 0\, ,$  and therefore $\delta_{\langle \nu_P , v_{P'}\rangle}^0=0$ in (\ref{comput2}). The sub-independency thus implies that $\widehat \mu_P (\ell_{p+1},...,\ell_n)=~0$ for every $P$\,. 
\end{itemize}
Therefore $d\mu_P (\theta_{p+1} \dots \theta_n)= \rho\, m_P\, d\theta_{p+1} \dots d\theta_n$ for any 
$P\in \mathcal C_p(V_{\mathbb T})$. This proves the strong extremality 
of $\mathscr T_n^p (V_\mathbb T)$ and ends the proof of the lemma.    
\end{proof}
\vskip 2mm
\noindent
\begin{proof}[Proof of Theorem \ref{maintheohigherdim}]
Let $V_{\mathbb{T}}$ be a strongly extremal tropical $p$-cycle. 
Let $P$ be a $p$-dimensional cell of the tropical $p$-cycle 
$V_{\mathbb T}$. The current $\mathscr T_P$ defined (in the preliminaries 
of Section \ref{tropicalcurrents}) as 
$$
\mathscr T_P:= (T_{H_P})_{|\Log^{-1} (\rm{int}(P))}
$$
coincides with the current 
$$
\int_{(\theta_{p+1},...,\theta_n) \in (\mathbb R/\mathbb Z)^{n-p}} 
[\Delta_{P,D_P,W} (\theta_{p+1},...,\theta_n)] \, d\theta_{p+1} \dots d\theta_n 
$$
about any point (in $(\mathbb C^*)^n$) which belongs to the $(n+p)$-dimensional real submanifold 
$\Log^{-1} ({\rm int} (P))$, where ${\rm int} (P)$ denotes the relative interior of $P$ in the affine $p$-plane 
$H_P$ (the argument is again the same as the one which has been invoked in the discussion preceding the Lemma 
\ref{actionofT}). About any point $\Log^{-1} (a)$, where $a$ lies in the relative interior (in the affine $(p-1)$-plane $H_W$) 
of a given facet $W$ of $P$, the normal current 
$\sum\limits_{W\prec P'} \mathscr T_{P'}$ coincides with the current 
$$
\sum\limits_{W\prec P'} \int_{(\theta_{p+1},...,\theta_n) \in (\mathbb R/\mathbb Z)^{n-p}} 
[\Delta_{P',D_{P'},W} (\theta_{p+1},...,\theta_n)] \, d\theta_{p+1} \dots d\theta_n, 
$$
By Theorem \ref{supportsmall}.
Since closedness of a current can be tested locally, it follows from the argument developed 
in the proof of Lemma \ref{mainlem-higherdim} that the current 
$$\mathscr T_n^p(V_{\mathbb T}) = \sum_{P \in \, \mathcal{C}_p(V_{\T})} \,m_P\mathscr{T}_P$$ 
is closed in a any compact neighborhood of any point $\Log^{-1}(a)$ in $\Log^{-1}({\rm int} (W))$, $W$ being an arbitrary facet 
of $P$, which in turn implies the closedness of $\mathscr{T}_n^p(V_{\T})$, noting that in light of Theorem \ref{supportsmall} we need not to check the closedness for faces of codimension higher that $1$. Suppose now that the for each facet $W\in \mathcal C_{p-1}(V_{\mathbb T})$, 
for each $P\in \mathcal C_p(V_\mathbb T)$ that 
shares $W$ as a facet, the projection along $W$ of primitive vectors $v_P^{[\rightarrow W]}$ 
form a linearly sub-independent set with cardinality $n-p+2$. 
Let $\widetilde{\mathscr T}\in \mathcal{D}'_{p,p}((\CC^*)^n)$ be a normal closed 
current with support $\Log^{-1}(V_\mathbb T)$. If $W\in \mathcal C_{p-1} (V_\mathbb T)$ and $a$ is a point in the relative interior of $W$ in $H_W$, the argument used in the proof of Lemma \ref{mainlem-higherdim} also shows that there exists some complex number 
$\rho_{W,a}$ such that, in neighborhood of $\Log^{-1}(a)$ in $(\C^*)^n$, one has 
$\widetilde {\mathscr T} = \rho_{W,a} \mathscr T_n^p(V_\mathbb T)$. Obviously, all $\rho_{W,a}$ (for $a$ in the relative interior of 
an arbitrary facet $W$ of $V_\mathbb T$) are equal to some complex number  
$\rho_W$. This implies that $\mathds{1}_{\Log^{-1}(\textrm{int} P)}\, \widetilde {\mathscr T} = \rho_{W}\, m_P \, \mathscr T_P\,.$ If $W'\not=W$ is another facet of $P$ we find a complex number $\rho_{W'}$ such that
$$\mathds{1}_{\Log^{-1}(\textrm{int} P)}\, \widetilde {\mathscr T} = \rho_{W'}\, m_P \, \mathscr T_P \,,$$
and $\rho_W = \rho_{W'}\,$ is imposed. Connectivity of $V_{\T}$ in codimension $1$ (as in the final step in the proof of Theorem \ref{maintheorem}) 
shows that all numbers $\rho_W$ ($W\in \mathcal C_{p-1} (V_\mathbb T)$) coincide (note that higher codimensional connectivity is not sufficient). This concludes the proof of the strong extremality of the current $\mathscr T_n^p(V_\mathbb T)$.            
\end{proof}

\section{Tropical currents in $\mathcal{D}'_{p,p}(\mathbb{CP}^n$)}

We first show that for a given effective tropical $p$-cycle $V_\mathbb T$ in $\RR^n$ the closed positive $(p,p)$-dimensional current $\mathscr T_n^p(V_\mathbb{T})$ (considered as a current in $(\C^*)^n$) can be extended by zero to a closed positive $(p,p)$-dimensional current in  $\mathbb{CP}^n$. 

\begin{lemma}\label{lemmaextension}
For any effective tropical $p$-cycle $V_{\mathbb{T}}$ in $\RR^n$, the positive tropical current $\mathscr T_n^p(V_{\mathbb{T}})\in \mathcal{D}'_{p,p}((\CC^*)^n)$ can be extended by zero to $\mathbb {CP}^n$ as a current $\bar{\mathscr T}_n^p (V_\mathbb T)$ in $\mathcal{D}'_{p,p}(\mathbb{CP}^n)$. Moreover, if $T \in \mathcal E^p((\CC^*)^n)$, then also 
$\bar{\mathscr T}_n^p (V_\mathbb T)\in \mathcal{E}^p(\mathbb{CP}^n)$.
\end{lemma}

\begin{proof}
Let $(\zeta_1,\dots,\zeta_n)$ be the coordinates on the complex torus $(\CC^*)^n$. 
Assume $P\in \mathcal C_p(V_\mathbb T)$, and without loss of generality that $0 \in \textrm{int} P\,.$ The current $T_{H_P}$ is expressed as the average 
$$
T_{H_P} = T_{H_P,D_P} = \int_{(\theta_{p+1},...,\theta_n) 
\in (\mathbb R/\mathbb Z)^{n-p}} 
[\Delta_{H_P,D_P} (\theta_{p+1},...,\theta_n)]\, d\theta_{p+1} \dots d\theta_n 
$$
(see formula \eqref{definitionofthecurrent}). 
For each $(\theta_{p+1},...,\theta_n) \in (\mathbb \R/\mathbb \Z)^{n-p}$, 
the complex $p$-dimensional analytic variety $\Delta_{H_P,D_P}(\theta_{p+1},...,\theta_n)$ is 
included in the toric subset of $(\mathbb C^*)^n$ defined in the coordinates 
$(\zeta_1,...,\zeta_n)$ by the set of binomial equations 
$$    
\prod_{j=1}^n \zeta_j^{\xi^+_{\ell,j}} - 
\prod\limits_{j=1}^n (\gamma_j(\theta,U_P,a))^{\xi_{\ell,j}}\,  \zeta_j^{\xi^-_{\ell,j}} =0,\quad \ell =1,...,M_{D_P},  
$$
where the $\xi_\ell = \xi_\ell^+ - \xi_\ell^-$ form a set of generators for 
${\rm Ker}_{B_P^t} \cap \mathbb \Z^n$ and 
$$
\gamma_j(\theta,U_P,a) = \exp \big(2i\pi (\theta_{p+1} u_{1,j}^P + \dots + 
\theta_n u_{n-p,j}^P)\big) \in \{\zeta \in \C^*\,;\, |\zeta|=1\},\quad j=1,...,n. 
$$ 
Each integration current 
$\Delta_{H_P,D_P}(\theta_{p+1},...,\theta_n)$ can then be extended to 
$\C \mathbb P^n$ as the integration of the Zariski closure (in $\C \mathbb P^n$) 
of the toric subset 
$$
\Big\{(\zeta_1,...,\zeta_n) 
\in (\mathbb C^*)^n\,;\, 
\prod_{j=1}^n \zeta_j^{\xi^+_{\ell,j}} - 
\prod\limits_{j=1}^n (\gamma_j(\theta,U_P,a))^{\xi_{\ell,j}}\, \zeta_j^{\xi^-_{\ell,j}} =0 \Big\}.  
$$
Since the degree of this projective algebraic variety is bounded independently of 
$(\theta_{p+1},...,\theta_n)$, the current $\mathscr T_n^p(V_\mathbb T)$ has finite mass 
about any point in $\C \mathbb P^n \setminus (\C^*)^n$.    
By the extension theorem of Skoda-El~Mir (see \cite{Dembook}, page 138), 
this current can then be trivially extended by $0$ as a positive $(p,p)$-dimensional closed 
current on $\mathbb{CP}^n$. The last assertion follows from the fact that $\mathscr T_n^p(V_\mathbb T)$ 
and $\bar{\mathscr T}_n^p(V_\mathbb T)$ have the same support in the dense open subset $(\CC^*)^n \subset \mathbb{CP}^n$\, and support of $\bar{\mathscr T}_n^p(V_\mathbb T)$ is the closure of support of ${\mathscr T}_n^p(V_\mathbb T)$ in $\mathbb{CP}^n$. 
\end{proof}

Let $d=d' + d''$ the usual decomposition of the de Rham (exterior) derivative and $d^c= (d' - d '')/(2i\pi)$, so 
that $dd^c =(1/i\pi)\, d''d'$. The following theorem gives a simpler representation for tropical currents of bidimension $(n-1,n-1)$ (equivalently of bidegree $(1,1)$). 

\begin{theorem}\label{thm-hyper-ddc}
Positive tropical currents of bidimension $(n-1,n-1)$ in $(\CC^*)^n$ (resp. their extension by zero to $\mathbb{CP}^n$) are exactly the currents of the form $dd^c [p\circ \Log]$, where $p$ is a tropical polynomial on $\RR^n$ (resp. $p$ is a homogeneous tropical polynomial on $\mathbb{TP}^n$).
\end{theorem}
\begin{proof}  
Observe that for the given tropical polynomial $p:\RR^n\rightarrow \RR$, the two positive closed $(n-1,n-1)$-dimensional currents $dd^c[p \circ \Log]$ and $\mathscr{T}^{n-1}_n(V_{\mathbb{T}}(p))$ share the same support, 
$\Log^{-1} (V_\mathbb T (p)),$ in $(\mathbb C^*)^n$. In order to show they coincide in $(\mathbb C^*)^n$, 
it is enough to prove they coincide in the open subset $\Log^{-1} \big(\R^n \setminus 
\bigcup_{\tau \in \mathcal C_{n-2} (V_\mathbb T)} |\tau|\big)$ (then they  
coincide in the whole $(\mathbb C^*)^n$ for dimensional reasons thanks to theorem 
\ref{supportsmall}). Since equality of currents can be tested locally, it is even enough 
to test such an equality in a neighborhood of $\Log^{-1}(a)$, where $a$ is an arbitrary point 
in the relative interior of a $(n-1)$-dimensional cell $P$ of the tropical hypersurface 
$V_\mathbb T(p)$. By a translation in $\mathbb R^n$, one can then assume that $p(x)= \max\{ \langle \alpha, x\rangle\,, 0 \}$, 
where $\alpha\in \mathbb Z^n\setminus \{(0,...,0)\}$. 
Let $\alpha = m_\xi\, \xi$, where $\xi$ is a primitive vector in 
$\mathbb Z^n$ and $m_{\xi}\in \mathbb N^*$. Let $B:=\{w_1,...,w_{n-1}\}$ be a 
$\Z$-basis for $H_P \cap \mathbb Z^n$ and consider a completion 
$D_B:=\{w_1,...,w_{n-1},u\}$ of $B$ as in the preliminaries of Section \ref{tropicalcurrents}. 
For each $\theta \in (\mathbb R/\mathbb Z)$, the toric set 
$\Delta_{H_P,D_B}(\theta)$ is the $(n-1)$-dimensional (reduced) toric hypersurface 
in $(\mathbb C^*)^n$ defined by the irreducible binomial 
$\prod_{j=1}^n \zeta_j^{\xi_j^+} - \gamma_u(\theta) 
\prod_{j=1}^n \zeta_j^{\xi_j^-}$ for some 
$\gamma_u (\theta)\in \mathbb S^1$ (see Remark \ref{remarktoricset}). 

Let $\overline {\Delta_{H_P,D_B}(\theta)}$ ($\theta \in \mathbb R/\mathbb Z$) be the 
Zariski closure of the hypersurface $\Delta_{H_P,D_B}(\theta)$ in 
$\mathbb C\mathbb P^n$, which is in fact the zero set in $\mathbb{CP}^n$ of homogenization of the above equation. 
It follows from Crofton's formula (see \cite{Dembook}, page 170, or Example (4.6) in \cite{Dem-gaz}) that  
$$
\deg \big(\overline{\Delta_{H_P,D_B} (\theta)}\big) = 
\max \big\{\sum\limits_{j=1}^n \xi_j^+, 
\sum\limits_{j=1}^n \xi_j^-\big\} =
\int_{\mathbb C \mathbb P^n}
\big[\overline {\Delta_{H_P,D_B}(\theta)}\big] 
\wedge \omega^{n-1}
$$
where $\omega$ denotes the K\"ahler form $\omega = dd^c \log \|\ \|$ in 
$\mathbb C\mathbb P^n$. 
On the other hand, it is easy to see that, in the weak sense of currents 
in $(\mathbb C^*)^n$, 
$$
\lim\limits_{m\rightarrow \infty} 
\frac{m_\xi}{m} \log \big| \prod_{j=1}^n \zeta_j^{m\xi_j} + 1\big| = p\circ \Log,  
$$  
which implies, taking $dd^c$, 
$$
\lim\limits_{m\rightarrow \infty} 
\frac{m_\xi}{m}\, dd^c \big[\log \big| \prod_{j=1}^n \zeta_j^{m\xi_j} + 1\big|\big] = 
dd^c\,[p\circ \Log]. 
$$ 
It follows that, if one denotes as $\overline{dd^c [p\circ \Log]}$ the trivial extension by $0$ of the 
positive closed $(n-1,n-1)$-dimensional current $dd^c [p\circ \Log]$ from 
$(\mathbb C^*)^n$ to $\mathbb C \mathbb P^n$, one has  
\begin{multline}\label{totalmass} 
\int_{\mathbb C\mathbb P^n} \overline{dd^c [p\circ \Log]} \wedge \omega^{n-1} = \\ 
\int_{\mathbb C\mathbb P^n} 
\Big[\int_{\mathbb R/\mathbb Z} 
\big[\overline {\Delta_{H_P,D_B}(\theta)}\big]\, d\theta\Big] 
\wedge \omega^{n-1} = \max \big\{\sum\limits_{j=1}^n \xi_j^+, 
\sum\limits_{j=1}^n \xi_j^-\big\}. 
\end{multline} 
Chose now $\xi' \in \mathbb Z^n\setminus \{(0,...,0)\}$ 
and a strictly increasing sequence $(N_k)_{k\geq 1}$ of 
positive integers such that all tropical $(n-1,n-1)$-hypersurfaces $V_\mathbb T (p_k)$, where 
$$
p_k~: x \in \mathbb R^n \mapsto \max \big\{p(x), \langle \xi',x \rangle - N_k\big\} 
= \max \big\{\langle \xi,x\rangle,0, 
\langle \xi',x\rangle - N_k\big\},\quad k\in \mathbb N^*,  
$$
are trivalent. For any relatively compact open subset $\mathcal V\subset \mathbb R^n$, 
$p\equiv p_k$ in $\mathcal V$ and the currents 
$dd^c [p\circ \Log]$ and $dd^c [p_k\circ \Log]$ coincide in $\Log^{-1} (\mathcal V)$ 
provided $k$ is large enough (depending on $\mathcal V$). Since the current 
$\mathscr T_n^{n-1} (V_{\mathbb T}(p_k))$ is extremal 
in $(\mathbb C^*)^n$ thanks to Theorem \ref{maintheohigherdim} ($p=n-1$), there exists, for each 
such $\mathcal V\subset \mathbb R^n$ and for any $k>>1$ 
large enough (depending on $\mathcal V$), a strictly positive constant 
$\rho_{\mathcal V,k}$ such that one has 
\begin{equation} 
\begin{split} 
& \big(\mathscr T_n^{n-1} (V_{\mathbb T} (p))\big)_{|\Log^{-1} (\mathcal V)}  
=\big(\mathscr T_n^{n-1} (V_{\mathbb T} (p_k))\big)_{|\Log^{-1} (\mathcal V)} = \\
& = \rho_{\mathcal V,k}\,  
\big(dd^c [p_k \circ \Log]\big)_{|\Log^{-1} (\mathcal V)}  
= \rho_{\mathcal V,k}\,  
\big(dd^c [p\circ \Log]\big)_{|\Log^{-1} (\mathcal V)}\,. 
\end{split} 
\end{equation} 
Taking an exhaustion of $\mathbb R^n$ with relatively open subsets 
$\mathcal V_\ell$, $\ell =1,2,...$, such that $\mathcal V_\ell \subset \mathcal V_{\ell +1}$ for any $\ell\in \mathbb N^*$, 
it follows that all $\rho_{\mathcal V,k}$ are equal,  so that there exists some strictly positive constant $\rho$ such that 
$$
\mathscr T_n^{n-1} (V_\mathbb T (p)) = \rho\, dd^c\, [p\circ \Log] 
$$
(as currents in $(\mathbb \C^*)^n$). The fact that the normalization constant 
$\rho$ equals $1$ follows from \eqref{totalmass} since 
$$
\mathscr T_n^{n-1} (V_\mathbb T (p)) = 
\int_{\theta \in (\mathbb R/\mathbb Z)} 
\big[ \Delta_{H_P,D_B} (\theta)]\, d\theta 
$$
(so that the trivial extensions of $\mathscr T_n^{n-1} (V_\mathbb T (p))$ and $dd^c \, [p\, \circ \Log]$ 
to $\mathbb C\mathbb P^n$ share the same total mass as currents in the projective space 
$\mathbb C\mathbb P^n$ equipped with its Fubini-Study K\"ahler form). 
\vskip 2mm
Regarding the statement for the homogeneous tropical polynomials, observe that extension by zero of $dd^c \,[{p} \circ \Log]$ to $\mathbb{CP}^n$, in the sense of currents, is exactly $dd^c~[\tilde{p}~\circ~\Log]\,$, where $\tilde{p}$ is the  homogenization of $p$\,.
\vskip 2mm
For the converse statement, just note that by Theorem \ref{tropicalhyp}, every tropical hypersurface $V_{\T}$ can be understood as $V_{\T}(p)$ for a tropical polynomial $p$, with equality of respective weights. 
\end{proof}

By previous theorem and Theorem \ref{maintheohigherdim} one readily has the following.  

\begin{corollary}\label{hypersurfacecase}
Let $p$ be a homogeneous tropical polynomial defining a tropical hypersurface in $\mathbb{TP}^n$. 
Then, the positive current $dd^c \, [p \circ \Log]$ is in $\mathcal{E}^{n-1}(\mathbb{CP}^n)$ if every facet of  the tropical hypersurface associated to $p$ is the common intersection of exactly $3$ polyhedra.  
\end{corollary}

\begin{remark}\label{rem-intersection}
\textnormal{
Let $V_1,\dots,V_n \subset \RR^n$ be the tropical hypersurfaces associated to $p_1,\dots,p_n: \RR^n \rightarrow \RR$, respectively. If these hypersurfaces intersect transversally, then the product
\begin{equation}\label{wedgeproducts}
\mathscr T_n^{n-1}(V_1) \wedge \dots \wedge \mathscr T_n^{n-1}(V_n)= dd^c\,[p_1 \circ \Log]\wedge \dots \wedge dd^c \, [p_n \circ \Log] 
\end{equation}
is well-defined. Using a formula due to A. Rashkovskii (see \cite{Rash}, \cite{Pass-Rull}) one has
\begin{equation}\label{rash-formula}
\frac{1}{n!}\int_{\Log^{-1}(E)}dd^c\,[p_1 \circ \Log]\wedge \dots \wedge dd^c \, [p_n \circ \Log] = \int_{E}\tilde{M}(p_1,\dots,p_n),
\end{equation}
where the right hand side is the real mixed Monge-Amp\`ere measure of the Borel set $E \subset \RR^n$ corresponding to the convex function $p_1,\dots, p_n.$
On the other hand for a tropical polynomial $p:\RR^n \rightarrow \RR $, one can calculate (as in \cite{yger} Example 3.20), 
$$
{M}(p):= \tilde{M}(p,\dots,p)= \sum_{a \in \,\mathcal{C}_0(V_{\mathbb{T}}(p))} {\rm Vol}_n(\{a \}^*)\delta_a\,,
$$
where $\delta_a$ denotes the Dirac mass at the vertex $a$ of $V_{\mathbb{T}}(p),$ and ${\rm Vol}_n(\{a \}^*)$ is the volume of the $n$-cell dual to $\{a\}$ in a dual decomposition of the Newton polytope of $p.$ This formula therefore, corresponds to a stable intersection, and also gives 
$$
\tilde{M}(p_1,\dots,p_n)= \sum_{\{a\} \in \,\mathcal{C}_0(V_1\cap \dots \cap V_n)} {\rm Vol}_n(\{a \}^*)\delta_a\,.
$$
From which the tropical B\'ezout's and Bernstein's theorems follow (see \cite{Bihan}, \cite{firststeps}).  We also remark that in view of (\ref{wedgeproducts}), (\ref{rash-formula}) relates the intersection of toric sets to intersection of tropical cycles. 
}
\end{remark}

\section{Amoebas and approximations of tropical currents}\label{approximationsection}
In this section we assume that all the tropical cycles are effective. 

\vskip 2mm 
\noindent
Suppose $\{X_t \}_{t\in \RR_{+}}$ is a family of algebraic cycles in $(\CC^*)^n$, $V_{\mathbb{T}}$ a tropical cycle, and 
$$
\lim_{t \rightarrow \infty} \Log(X_t) = V_{\mathbb{T}}
$$
with respect to the Hausdorff metric on compact sets of $\RR^n$. As we have mentioned in Section \ref{tropicalsection}, if such a limit exists, this limit inherits the structure of a tropical cycle, the weights of this  tropical cycle being related to the degrees of the cycles $X_t$. The approximation problem considered in \cite{brugalle-shaw},\cite{katz},\cite{speyer thesis},\cite{mb} deals with approximation of a tropical cycle or a tropical curve by  algebraic varieties $X_t$ with equal total degrees. However we are interested here in the problem of approximations of the tropical cycles as sets, which makes the approximation more flexible. For instance, Grigory Mikhalkin's example of a spatial tropical cubic (\cite{Mikh-trop-enum}, \cite{speyer thesis}) of genus $1$ is not approximable by amoebas of cubic curves in $(\CC^*)^3$ but, as a set, it is approximable by a family of sextic curves, the resulting sextic tropical curve being the cubic tropical curve with doubled weights (see \cite{moi}). As for the theory of currents, it is important to know if there are elements of $\mathcal{E}^k(X)$ (with Hodge classes) which are not in $\overline{\mathcal{I}^k(X)}~$, where $X$ is a projective variety (see \cite{DemEx}, \cite{Dem-analytic}). In this section we relate the approximation problem of tropical cycles by amoebas to the problem of approximating tropical currents by integration currents along analytic cycles with positive coefficients. 

\begin{definition}\label{set-wise-def} \rm
We call a tropical cycle set-wise approximable if its underlying set is approximable in Hausdorff metric by amoebas of algebraic varieties of any degree. 
\end{definition}

\begin{remark}
\textnormal{
The set-wise approximability for strongly extremal tropical cycles is equivalent to having a multiple (obtained by multiplying the weights) which is approximable by amoebas of algebraic subvarieties of $(\CC^*)^n$ with equal degrees. 
}
\end{remark}

\begin{example}
{\rm Consider the tropical polynomial
$$p:\RR^n \rightarrow \RR, \quad x=(x_1,\dots,x_n) \mapsto \max_{\alpha} \big\{c_{\alpha}+ \alpha_1 x_1+\dots+\alpha_n x_n\big\}
$$ 
attached to a finite set of indices $\alpha= (\alpha_1,\dots,\alpha_n)\in (\mathbb{Z}_{\geq 0})^n$. Now, for each $m,l \in \mathbb{N}^*$,  consider the polynomial map  
$$
f_{l,m}: \CC^n \rightarrow \CC, \quad z=(z_1,\dots,z_n)\mapsto \sum_{\alpha} \exp({l \,c_{\alpha}}) z_1^{m \alpha_1}\dots z_n^{m \alpha_n}.
$$
It is easy to see that, in the sense of distributions,  
$$
\lim_{m \rightarrow \infty} \frac{1}{m} \, \log|f_{m,m}(z)|  =  p \circ \Log (z).
$$  
Poincar\'e-Lelong equation (\cite{Dembook}, page 143) yields (in the sense of currents) 
\begin{equation}\label{pioncareloleng}
\lim_{m \rightarrow \infty} \frac{1}{m}[Z_{f_{m,m}}]= dd^c [ p \circ \Log(z)],
\end{equation}
where $Z_{f_{m,m}}$ denotes the divisor of $f_{m,m}$ with multiplicities taken into account. Moreover
$$
\lim_{m\rightarrow \infty}\Log (\text{{\rm Supp} $[Z_{f_{m,m}}]$}) = \lim_{m\rightarrow \infty} \mathcal{A}_{f_{m,m}}= \lim_{m\rightarrow \infty}\frac{1}{m} \mathcal{A}_{f_{m,1}}= V_{\mathbb{T}}(p),
$$
where in the third equation ``multiplying'' the amoeba $\mathcal{A}_{f_{m,1}}$ by $1/m$ means dilating this amoeba by this factor. Therefore, the support of the currents on the left hand side of (\ref{pioncareloleng}) approximates the support of the current on the right hand side and the coefficient $\frac{1}{m}$ makes the total masses equal. Combining this with Theorem \ref{thm-hyper-ddc} one has 
\begin{equation}\label{poincare-lelong-tropical}
\lim_{m \rightarrow \infty} \frac{1}{m}[Z_{f_{m,m}}]= dd^c [p \circ \Log(z)] = \mathscr{T}_n^{n-1}(V_{\T})\,.
\end{equation}
}
\end{example}
 
Assume now that $V_{\mathbb{T}} \subset \RR^n$ is an effective strongly extremal tropical $p$-cycle. Suppose next that there is a family of algebraic $p$-cycles $(X_t)_{t>1}$ in 
$(\mathbb C^*)^n$ such that we have the set-wise approximation  
\begin{equation}\label{assumptionloglim}
\lim_{t \rightarrow \infty} \Log_t(X_t) = V_{\mathbb{T}}\,,
\end{equation}
where $\Log_t (z_1,\dots,z_n) := (\log |z_1|^{1/\log t},\dots,\log |z_n|^{1/\log t})$ for $t>1$. 
Starting with such a set-wise approximation, we intend to find a sequence of integration currents 
$\mathcal I^p((\C^*)^n)$ that converges to a multiple of $\mathscr T_n^p(V_\mathbb T)$.
\vskip 2mm 
For every positive integer $m$, define the proper smooth map 
\begin{equation}
\begin{split}
\Phi_m: \CC^n \rightarrow \CC^n\ ,
\ (z_1,\dots,z_n) \mapsto (z_1^m,\dots,z_n^m)
\end{split}
\end{equation}
and consider the current integration current $\Phi_m^* [X_t]:= [\Phi_m^{-1}(X_t)]$. The support of this current 
is obviously the set
\begin{equation}\label{7.22}
\begin{split}
&\Phi_m^{-1}(X_t)= 
\big\{(w_1,\dots,w_n)\in (\mathbb C^*)^n\,;\, (w_1^m,\dots,w_n^m) \in X_t \big\} \\
&\quad =\quad \Big\{\Big(\exp\big(\frac{2 \pi i k_1+\arg(z_1)}{m}\big)|z_1|^{1/m},\dots,
\exp\big(\frac{2 \pi i k_n+\arg(z_n)}{m}\big)|z_n|^{1/m} \Big) \ , \\ & \quad\quad \quad\quad\quad\quad\quad\quad\quad\quad\quad\quad\quad\quad\quad \quad\quad\quad (z_1,\dots,z_n) \in X_t \ , 0\leq k_j\leq m-1 \Big\}.
\end{split}
\end{equation}
\vskip 1mm
\noindent
Note that as $m$ increases, the set $\big\{e^{2 \pi i k/m}, k=0,\dots,m-1 \big\}$  tends to a dense set in the unit circle $\mathbb S^1$. Let $m~: [1,\infty[\rightarrow \N$ be an increasing function 
tending to infinity when $t$ tends to infinity. 
Therefore the support of a limit current for any convergent sequence 
of the form $\big(\lambda_{m(t_{k})} [\Phi_{m(t_{k})}^{-1} (X_{t_{k}})/\deg X_t]\big)_k$ 
such that $(t_k)_k$ tends to $+\infty$, is necessarily of the form $\Log^{-1}(V)$ for some closed 
set $V\subset \R^n$.     
\vskip 1mm
\noindent
On the other hand, if $x=(x_1,\dots,x_n) \in V_{\mathbb{T}}$, then there exists a sequence of points 
$$
\Big(\zeta_{t_{\nu_k}} = (\zeta_{t_{\nu_k},1},\dots,\zeta_{t_{\nu_k},n}) \in C_{t_{\nu_k}}\Big)_k 
$$
such that 
$$
\Log_{t_{\nu_k}}(\zeta_{t_{\nu_k}}) \rightarrow x,
$$
or
$$(|\zeta_{t_{\nu_k},1}|^{1/\log t_{\nu_k}},\dots,|\zeta_{t_{\nu_k},n}|^{1/\log t_{\nu_k}}) \rightarrow (e^{x_1},\dots,e^{x_n})$$  
as the sub-sequence $(\nu_k)_k = (\nu_k(x))_k$ tends to $+\infty$. Comparing this with (\ref{7.22}), if one takes $m~: t\in [1,+\infty[ \mapsto [\log t]$, the integer part of $\log t$, then the support of a limit current for any convergent sequence 
of the form $\big(\lambda_{m(t_{k})} [\Phi_{m(t_{k})}^{-1} (X_{t_{k}})]\big)_k$ 
such that $(t_k)_k$ tends to $+\infty$ equals necessarily to $V_{\mathbb{T}}$\,. 
If one takes $\lambda_{m} = m^{n-p}$ the family of currents  
$$
\frac{1}{(m(t))^{n-p}}\, \frac{1}{\deg X_t} \, \big[\Phi_{m(t)}^* [X_t]\big],\qquad t>1 
$$
is normalized (with degrees all equal to $1$). 
Thanks to Theorem \ref{maintheorem}, any subsequence of it converges towards the same multiple 
$\lambda \mathscr T_n^p(V_\mathbb T)$ ($\lambda >0$) of the extremal current $\mathscr{T}_n^p(V_{\mathbb{T}})$.  
So we have proved the following.  

\begin{theorem}\label{current-amoeba-approx}
Assume that the tropical cycle $V_{\mathbb{T}}$ is strongly extremal and set-wise approximable as $\lim_{t\rightarrow +\infty} \Log_t (X_t)$ by amoebas 
of irreducible algebraic $p$-cycles $(X_t)_{t>1}$ of $(\CC^*)^n$. Then there exists $\lambda> 0$ such that
$$
\mathscr T_n^p(V_\mathbb T) = \lambda\lim_{m \rightarrow \infty} \frac{1}{m^{n-p}} \Phi_m^*[X_{e^m}].
$$
In particular, $\mathscr T_n^p(V_\mathbb T) \in \overline{\mathcal{I}^p\big((\CC^*)^n \big)}$. 
\end{theorem}

\begin{remark}
{\rm Let $V_{\mathbb{T}}$ be an effective tropical $p$-cycle. 
By Theorem \ref{lemmaextension} the current  $\mathscr T_n^p(V_\mathbb T)
\in SPC^p((\C^*)^n)$ can be extended by zero to $\bar{\mathscr T}_n^p(V_\mathbb T)\in SPC^p(\mathbb{CP}^n).$ As a result, if in the above theorem one approximates $V_{\mathbb{T}}$ by amoebas of irreducible algebraic cycles (= analytic cycles by Chow's theorem) of $\mathbb{CP}^n$ which do not lie entirely in $\{z_0\cdots z_n=0\}$ then the theorem also
gives $\bar{\mathscr T}_n^p(V_\mathbb T)\in  \overline{\mathcal{I}^p\big(\mathbb{CP}^n \big)}$.}
\end{remark}

The above discussion highlights the following important questions. 
\begin{problem}
{\rm Are there strongly extremal tropical cycles which are not set-wise approximable?} 
\end{problem} 

\begin{problem}
[Converse of Theorem \ref{current-amoeba-approx}] 
{\rm Assume $V_{\mathbb{T}}$ is a tropical $p$-cycle such that $\mathscr {T}_n^p(V_\mathbb{T})$ is extremal. 
Does $\mathscr T_n^p(V_\mathbb T)\in \overline{\mathcal{I}^p\big((\CC^*)^n \big)}$ imply that $V_{\mathbb{T}}$ is set-wise approximable by amoebas of algebraic varieties in $(\CC^*)^n$?} 
\end{problem}

\begin{problem}
{\rm How one could generalize these constructions to ``infinite" tropical cycles?} 
\end{problem}
\subsection*{Acknowledgements}
I am grateful to Jean-Pierre Demailly for his hospitality in Grenoble, for the invaluable discussions, and for his guidance in explicitly resolving the extremality of $dd^c\max\log~\{1,2|z_1|,3|z_2|\}$.  I am also grateful to Alain Yger for useful discussions, and for enhancing the exposition of the notes as well as to Erwan Brugall\'e for reading an early draft of this work, and for his comments on tropical geometry, especially the set-wise approximations of the spatial cubics. I also thank Vincent Koziarz for reading an early draft, and his question/comment on foliations.

\
\noindent
F. Babaee, Institut de Mathematiques, Universit\'e de Bordeaux,  33405 Talence,  France. \texttt{farhad.babaee@math.u-bordeaux1.fr}


\begin{thebibliography}{}



\bibitem{moi}{F. Babaee}~: \textit{Complex tropical currents}, PhD thesis, in preparation. 

\bibitem{Bihan}{B. Bertrand, F. Bihan}~: \textit{Euler characteristic of real nondegenerate tropical complete intersections}, Preprint, \texttt{arXiv:0710.1222v2} 


\bibitem{mb}{B. Bertrand, E. Brugall\'e, G. Mikhalkin}~: \textit{Genus 0 characteristic numbers of tropical projective plane,} to appear in Compositio Mathematica. Available at {\tt http://erwan.brugalle.perso.math.cnrs.fr/publications.html.}

\bibitem{brugalle-shaw}{E. Brugall\'e , K. Shaw}~: \textit{Obstructions to approximating tropical curves in surfaces via intersection theory,} Preprint, {\tt arXiv:1110.0533v1}.

\bibitem{Dem-gaz}{J.-P. Demailly}~: \textit{Courants positifs et th\'eorie de l'intersection}, Gaz. Math. {\bf 53} (1992) 131-159.

\bibitem {Dem-analytic}{J.-P. Demailly}~: \textit{Analytic methods in algebraic geometry}, {Surveys of Modern Mathematics}, {\bf 1}, {International Press}, {Somerville, MA}, {2012}, ISBN :{978-1-57146-234-3}.

\bibitem{Dembook} {J.-P. Demailly}~: {\it Complex analytic and differential geometry}, Free book available at {\tt http://www.fourier.ujf-grenoble.fr/$\sim$demailly/books.html.}

\bibitem{DemEx}{{J.-P. Demailly}~: \textit{Courants positifs extr\^emaux et conjecture de Hodge}, Invent. Math. {\bf 69} (1982), no. 3, 347-374.}

\bibitem{Demailly1}{J.-P. Demailly}~: {\it Regularization of closed positive currents and intersection theory,} { J. Algebraic Geom.} {\bf  1} (1992), no. 3, 361-409.

\bibitem{Dinh-Sibony}{T. C. Dinh}, {N. Sibony}~: \textit{Dynamics of regular birational maps in $\Bbb P^k,$} J. Funct. Anal. {\bf 222} (2005), no. 1, 202–216.

\bibitem{sib-dinh-rigidity}{T. C. Dinh}, {N. Sibony}~: \textit{Rigidity of Julia sets for Henon type maps,} Preprint, {\tt arXiv:1301.3917v1.} 

\bibitem{Gelfand}{{I. M. Gelfand,} {M. M. Kapranov,} {A. V. Zelevinsky}~: \textit{Discriminants, resultants and multidimensional determinants}, {Modern Birkh\"auser Classics}, {Reprint of the 1994 edition},{Birkh\"auser Boston Inc.},{Boston, MA},{2008}, ISBN : {978-0-8176-4770-4}.}

\bibitem{geomofnumbers}{Gruber, P. M.}, {Lekkerkerker, C. G.}~: \textit{Geometry of numbers}, {North-Holland Mathematical Library}, \textbf{37}, edition {2}, {North-Holland Publishing Co.}, {Amsterdam}, {1987}.


\bibitem{Guedj}{V. Guedj}~: {\it Courants extr\'emaux et dynamique complexe,} Ann. Sci. \'Ecole Norm. Sup. (\textbf{4}) 38 (2005), no. 3, 407–426. 


\bibitem{katz} E. Katz~: \textit{Lifting tropical curves in space and linear systems on graphs,} Adv. Math. {\bf 230} (2012), no. 3, 853–875. 


\bibitem{Lagerberg}{A. Lagerberg~: \textit{Super currents and tropical geometry,} 
Math. Z. {\bf 270} (2012), no. 3-4, 1011–1050.}



\bibitem{Lelong}{P. Lelong~: {\it \'El\'ements extr\'emaux sur le c\^one des courants positifs ferm\'es,} (French) S\'eminaire Pierre Lelong (Analyse), Ann\'ee 1971-1972 pp. 112–131. Lecture Notes in Math., \textbf{332}, Springer, Berlin, 1973.}  

\bibitem{Mikh-amoeba} G. Mikhalkin~: \textit{Amoebas of algebraic varieties}, a report for the Real Algebraic and Analytic Geometry Congress, June, 2001, Rennes, France, \texttt{arXiv: math.AG/0108225.}

\bibitem{Mikh-pants} G. Mikhalkin~: \textit{Decomposition into pairs-of-pants for complex algebraic hypersurfaces,} Topology, {\bf 43} (2004), no. 5, 1035-1065. 

\bibitem{Mikh-trop-enum} G. Mikhalkin~: \textit{Enumerative tropical geometry in $\RR^2$}, J. Amer. Math. Soc. {\bf 18} (2005), 313-377, \texttt{arXiv: math.AG/0312530.}

\bibitem{Mikh-katz-notes} G. Mikhalkin~: \textit{Tropical geometry}, Notes by E. Katz, Available at
\texttt{https://www.ma.utexas.edu/rtgs/geomtop/rtg/notes/Mikhalkin\_Tropical\\\_Lectures.pdf}.


\bibitem{Mikh-trop-appli} G. Mikhalkin~: \textit{Tropical geometry and its applications}, International Congress of Mathematicians. Vol. II, pages 827-852. Eur. Math. Soc., Z\u"rich, 2006.


\bibitem{Pass-Rull} M. Passare, H. Rullg\aa rd~: \textit{Amoebas, Monge-Amp\`ere measures, and triangulations of the Newton polytope}, Duke Math. Journal, \textbf{121}, 3 (2004), 481-507.

\bibitem{Rash} A. Rashkovskii~: \textit{Indicators for plurisubharmonic functions of logarithmic growth,} Indiana Univ. Math. J. \textbf{50} (2001), no. 3, 1433- 1446. 


\bibitem{firststeps}J. Richter-Gebert, B. Sturmfels, T. Theobald~: {\it First steps in tropical geometry,} Idempotent mathematics and mathematical physics, 289–317, Contemp. Math., \textbf{377}, Amer. Math. Soc., Providence, RI, 2005. 

\bibitem{shaw}{K. Shaw}~: {\it Tropical Intersection Theory and Surfaces,} PhD thesis, available at \texttt{http://www.math.toronto.edu/shawkm/theseShaw.pdf}


\bibitem{sibony}{N. Sibony}~: \textit{Dynamique des applications rationnelles de $\bold P^k$,} Dynamique et g\'eom\'etrie complexes (Lyon, 1997), ix–x, xi–xii, 97–185, 
Panor. Synth\'eses, \textbf{8}, Soc. Math. France, Paris, 1999. 


\bibitem{speyer thesis}{D. Speyer~: \textit{Tropical geometry,} PhD Thesis, available at \texttt{http://www-personal.umich.edu/$\sim$speyer/}}.

\bibitem{sturm-toricset}B. Sturmfels~: \textit{Gr\"obner Bases and Convex Polytopes,} University Lecture Series, Vol. \textbf{8}, American Mathematical Society, Providence, RI, 1995.

\bibitem{yger}{A. Yger}~: \textit{Tropical geometry and amoebas,} Lecture notes available at \texttt{http://cel.archives-ouvertes.fr/cel-00728880}.


\end{thebibliography}
\end{document}